\documentclass[11pt,a4paper]{amsart}
\RequirePackage{pdf14}
\RequirePackage{fix-cm}
\usepackage[left=2.5cm, right=2.5cm,top=3.5cm]{geometry} %for editing
\usepackage[utf8]{inputenc}
\usepackage{graphicx}
\usepackage{relsize}
\usepackage{mathrsfs,amsmath,amscd,amsthm,amssymb,verbatim,mathtools}
\usepackage{xcolor}
\usepackage{todonotes}
\usepackage{bm,bbm}
\usepackage{parskip}
\usepackage[shortcuts]{extdash}
\numberwithin{equation}{section}
\numberwithin{figure}{section}
\numberwithin{table}{section}

\usepackage{amsfonts}
\definecolor{e-mail}{rgb}{0,.40,.80}
\definecolor{reference}{rgb}{.20,.60,.22}
\definecolor{citation}{rgb}{0,.40,.80}

\usepackage[colorlinks=true,
linkcolor=reference,
citecolor=citation,
urlcolor=e-mail]{hyperref}
\usepackage[capitalise]{cleveref}

\usepackage[all,cmtip]{xy}

\usepackage{tikz,tikz-cd}
\usetikzlibrary{arrows,shapes}
\usetikzlibrary{trees}
\usetikzlibrary{matrix,arrows}
\usetikzlibrary{positioning}
\usetikzlibrary{calc,through}
\usetikzlibrary{decorations.pathreplacing}
\usepackage{pgffor}

\usetikzlibrary{decorations.pathmorphing}
\usetikzlibrary{decorations.markings}
\tikzset{
	vector/.style={decorate, decoration={snake}, draw},
	provector/.style={decorate, decoration={snake,amplitude=2.5pt}, draw},
	antivector/.style={decorate, decoration={snake,amplitude=-2.5pt}, draw},
	fermion/.style={draw=black, postaction={decorate},
		decoration={markings,mark=at position .55 with {\arrow[draw=black]{>}}}},
	fermionbar/.style={draw=black, postaction={decorate},
		decoration={markings,mark=at position .55 with {\arrow[draw=black]{<}}}},
	fermionnoarrow/.style={draw=black},
	gluon/.style={decorate, draw=black,
		decoration={coil,amplitude=4pt, segment length=5pt}},
	scalar/.style={dashed,draw=black, postaction={decorate},
		decoration={markings,mark=at position .55 with {\arrow[draw=black]{>}}}},
	scalarbar/.style={dashed,draw=black, postaction={decorate},
		dwecoration={markings,mark=at position .55 with {\arrow[draw=black]{<}}}},
	scalarnoarrow/.style={dashed,draw=black},
	electron/.style={draw=black, postaction={decorate}/,
		decoration={markings,mark=at position .55 with {\arrow[draw=black]{>}}}},
	bigvector/.style={decorate, decoration={snake,amplitude=4pt}, draw},
}

\newtheorem{thm}{Theorem}[section]
\newtheorem{prop}[thm]{Proposition}
\newtheorem{lem}[thm]{Lemma}
\newtheorem{cor}[thm]{Corollary}

\theoremstyle{definition}
\newtheorem{dfn}[thm]{Definition}
\newtheorem{dfn/lem}{Definition/Lemma}

\theoremstyle{remark}
\newtheorem{rmk}[thm]{Remark}
\newtheorem{physrmk}[thm]{Physical Remark}

\newcommand{\defterm}[1]{\textbf{\emph{#1}}}

\def\be{\begin{equation}}
	\def\ee{\end{equation}}

\DeclareMathOperator{\Tr}{Tr}

\newcommand{\ii}{\text{i}}

\newcommand{\id}{\text{id}}
\newcommand{\pd}{\partial}
\newcommand{\opd}{\overline{\partial}}
\newcommand*\diff{\mathop{}\!\mathrm{d}}

\newcommand{\wt}{\widetilde}
\newcommand{\wh}{\widehat}

\newcommand{\fsl}{\mathfrak{sl}}

\newcommand{\fg}{\mathfrak{g}}

\newcommand{\R}{{\mathbb R}}
\newcommand{\N}{{\mathbb N}}
\renewcommand{\P}{{\mathbb P}}

\newcommand{\C}{{\mathbb C}}
\newcommand{\Z}{{\mathbb Z}}

\newcommand{\CF}{{\mathcal F}}

\newcommand{\CN}{{\mathcal N}}
\newcommand{\CO}{{\mathcal O}}

\newcommand{\CR}{{\mathcal R}}
\newcommand{\CS}{{\mathcal S}}

\newcommand{\CV}{{\mathcal V}}

\newcommand{\bQ}{\mathbf{Q}}

\begin{document}

    %!TEX Root = ./formality.tex

\title[Semi-infinite cohomology of graded-unitary vertex algebras]{On the semi-infinite cohomology of graded-unitary vertex algebras}

\author[C. Beem]{Christopher Beem}
\author[N. Garner]{Niklas Garner}

\address{Mathematical Institute, University of Oxford, Woodstock Road, Oxford, OX2 6GG, United Kingdom}

\email{christopher.beem@maths.ox.ac.uk, niklas.garner@maths.ox.ac.uk}

\begin{abstract}
    Recently, the first author with A. Ardehali, M. Lemos, and L. Rastelli introduced the notion of \emph{graded unitarity} for vertex algebras. This generalization of unitarity is motivated by the SCFT/VOA correspondence and introduces a novel Hilbert space structure on the state space of a large class of vertex algebras that are not unitary in the conventional sense. In this paper, we study the relative semi-infinite cohomology of graded-unitary vertex algebras that admit a chiral quantum moment map for an affine current algebra at twice the critical level. We show that the relative semi-infinite chain complex for such a graded-unitary vertex algebra has a structure analogous to that of differential forms on a compact K\"ahler manifold, generalizing a strong form of the classic construction of Banks--Peskin and Frenkel--Garland--Zuckerman. We deduce that the relative semi-infinite cohomology is itself graded-unitary, which establishes graded unitarity for a large class of vertex operator algebras arising from three- and four-dimensional supersymmetric quantum field theories. We further establish an outer ${\rm USp}(2)$ action on the semi-infinite cohomology (which does not respect cohomological grading), analogous to the Lefschetz $\mathfrak{sl}(2)$ in K\"ahler geometry. We also show that the semi-infinite chain complex is quasi-isomorphic as a differential graded vertex algebra to its cohomology, in analogy to the formality result of Deligne--Griffiths--Morgan--Sullivan for the de Rham cohomology of compact K\"ahler manifolds. We conclude by observing consequences of these results to the associated Poisson vertex algebras and related finite-type derived Poisson reductions.
\end{abstract}

\maketitle
\tableofcontents
    \setlength{\parskip}{10pt}
    %
    %!TEX Root = ../formality.tex

%-------------------------------------%
\section{\label{sec:intro}Introduction}
%-------------------------------------%

Vertex algebras arise in a variety of physical systems that possess observables whose correlations depend holomorphically or meromorphically on their positions on a plane (or Riemann surface) embedded in spacetime. Depending on the precise physical setting in which they originate, they may be equipped with additional structures. For example, when there is a conserved current implementing translations, the vertex algebra possesses a stress tensor/conformal vector, leading to a vertex operator algebra (VOA), or a compatible (shifted) vertex Lie structure if there are additional topological directions transverse to the vertex algebra plane leading to a (shifted) Poisson vertex algebra.

The vertex algebras of principle interest in the present work arise in the context of the four-dimensional $\CN = 2$ superconformal field theories (SCFTs) by means of the construction of \cite{Beem:2013sza}, where the vertex algebra is realized by a collection of local observables called (twisted-translated) \emph{Schur operators} \cite{Gadde:2011uv}, restricted to a preferred plane in the four-dimensional spacetime. Meromorphicity follows from a preserved superconformal symmetry.%
\footnote{These vertex algebras generally have both Grassmann even and odd elements and are more properly referred to as vertex superalgebras. We will omit the prefix ``super'' throughout this work, but it should be understood that in general our vertex algebras may have Grassmann odd parts.} %
These vertex algebras inherit several pieces of additional structure from their four-dimensional parents. For example, they have a canonical stress tensor/conformal vector which, as a consequence of four-dimensional unitarity, necessarily has a negative Virasoro central charge. Additionally, the vector space of Schur operators is triply-graded:
%%%%%%
\[
    \CV = \bigoplus_{\substack{h \in \scriptstyle{\frac{1}{2}}\N\\ R \in \scriptstyle{\frac{1}{2}}\N\\ d \in \Z}} \CV_{h,R,d} 
\]
%%%%%%
where $h$ is the conformal weight grading, corresponding physically to a sum of four-dimensional spin and ${\rm SU}(2)_R$ weights $h = j_1 + j_2 + R$, and $d$ is an additional grading that is cohomological in nature, corresponding physically to twice the ${\rm U}(1)_r$ charge $d = 2r$. All the conformal weight spaces are finite-dimensional, and the only Schur operator with $h = 0$ is the identity operator, so these are VOAs \emph{of CFT type}%
\footnote{In fact, they are of \emph{strong} CFT type in the sense of \cite{DongMason2004}, see also \cite{Moriwaki2020689}.} %
equipped with an additional cohomological $\Z$ grading that is unrelated to conformal weight, which we will call the \emph{internal} $\mathbb{Z}$ grading. On the other hand, the grading by $R$, corresponding to weights with respect to the ${\rm SU}(2)_R$ symmetry in four dimensions, does not lead to a vertex-algebraic grading and instead leads to a filtration \cite{Beem:2017ooy} that is good in the sense of Li \cite{Li2004}.

The collection of such VOAs is remarkably rich, but does not share many of the features familiar to the study of two-dimensional (chiral) conformal field theories. For example, they are never unitary, as witnessed by, \emph{e.g.}, their necessarily negative Virasoro central charges, and they are generally logarithmic. However, while not unitary in the ordinary sense, recent work \cite{ArabiArdehali:2025fad} of the first author with A. Ardehali, M. Lemos, and L. Rastelli established that the unitarity of the parent four-dimensional SCFT is inherited by the associated VOA in a more subtle form. This structure was dubbed \emph{graded unitarity} in \emph{loc. cit.} In brief, the vector space of local operators in the parent four-dimensional theory is equipped with a natural Hermitian form, coming from suitable two-point functions%
\footnote{In \cite{ArabiArdehali:2025fad}, two-point functions were considered in $\R^4$ with operators placed at Euclidean times $t_E = \pm \frac12$, and with conjugation determined in a planar quantization scheme. One may equivalently consider two-point functions on $S^4$ with operators placed at the poles using the conjugation operation of radial quantization, which is more natural from the perspective of the state-operator correspondence. These two conventions are related by a global conformal transformation and are equivalent.}, %
and reflection positivity of the underlying SCFT guarantees that it is positive-definite. This inner product assigns a norm to a Schur operator $\CO$ by pairing it with its conjugate $\CO^\dagger$, but the latter is \emph{not} a Schur operator and therefore this inner product is not transparently visible in the vertex algebra. Nonetheless, $\CO^\dagger$ is in the same ${\rm SU}(2)_R$ multiplet as another Schur operator $\rho(\CO)$; the graded-unitary structure for $\CV$ is defined in such a way that the resulting inner product assigns a positive norm to $\CO$ by pairing it with $\rho(\CO)$ and accounting for additional kinematical phases.

The existence of a graded-unitary structure on a vertex algebra imposes many interesting consistency conditions; the work \cite{ArabiArdehali:2025fad} studied some of these constraints in the cases where the underlying vertex algebra is (the simple quotient of) a Virasoro vertex algebra $M(p,q)$ or an affine current algebra $L_k(\fsl(N))$ for small $N$, albeit without confirming the existence of the full graded-unitary structure for these examples. Nevertheless, by making certain physically-motivated assumptions on the form of the $\mathfrak{R}$-filtration%
\footnote{Concretely, that work assumes that the physically-relevant filtrations of the Virasoro vertex algebras are weight-based filtrations coming from assigning the conformal vector weight one. The filtration assumed for $L_k(\fsl(2))$ is similarly a weight-based filtration where the generating currents and Sugawara conformal vector are all given weight one; the proposal for the higher-rank analogues of this filtration give the currents weight one and ``Casimir'' vectors a weight of one less than their degree.}: %
the only allowed $M(p,q)$ are of the form $M(2,2n+3)$, and the only allowed levels $k$ for $L_k(\fsl(2))$ and $L_k(\fsl(3))$ are precisely the boundary admissible levels. The $M(2,2n+3)$ algebras are realized by the $(A_1, A_{2n})$ Argyres-Douglas theories in four dimensions, and the $L_{-2+\scriptstyle{\frac{2}{q}}}(\fsl(2))$ are realized by the $(A_1, D_q)$ theories \cite{Cordova:2015nma}, so it expected that these examples to admit graded-unitary structures. For other simple $\fg$, the results of \cite{ArabiArdehali:2025fad} indicate that there could be additional levels $k$ beyond the boundary admissible levels that are compatible with a graded-unitary structure on $L_k(\fg)$, but a more detailed analysis was left for future work.

We describe the mathematical notion of a graded-unitary vertex algebra in Section \ref{sec:unitary}---slightly refining the definition of \cite{ArabiArdehali:2025fad}---as well as a slightly weakened version thereof. We prove several elementary properties of weak graded-unitary vertex algebras, such as a shortening condition on the $\mathfrak{R}$-filtration in Proposition \ref{prop:shortening} and Graded-Unitary Spin-Statistics in Proposition \ref{prop:spinstatistics}, and establish graded unitarity for the simple examples given in Section 2.4 of \cite{ArabiArdehali:2025fad}. While graded unitarity has its origins in four-dimensional physics, our weakened version thereof applies to a much wider range of vertex algebras and includes genuinely unitary vertex algebras, as described in \emph{e.g.} \cite{DongLin2014, AiLin2017, Kac:2020ytv, RaymondTanimotoTener2022}, see Section \ref{sec:ex0}, as well as products of unitary and graded-unitary vertex algebras that arise in other physical and mathematical contexts.

The main purpose of the present work is to establish several structural properties concerning the relative semi-infinite cohomology of weak graded-unitary vertex algebras. Our primary interest is in the relative semi-infinite cohomologies that arise from four-dimensional superconformal gauge theories, but our results apply more generally, including, \emph{e.g.}, the vertex algebras associated to hypertoric varieties studied by Kuwabara \cite{Kuwabara2021} and many of those appearing on the boundary of $A$-twisted three-dimensional $\CN=4$ gauge theories \cite{Costello:2018fnz}.
We begin Section \ref{sec:BRST} with a brief review of relative semi-infinite $\wh{\fg}_{-2h^\vee}$-cohomology with coefficients in a vertex algebra $\mathcal{V}_{\rm matter}$ with a Hamiltonian $\wh{\fg}_{-2h^\vee}$ action, \emph{i.e.} admitting a $\widehat{\mathfrak{g}}$ chiral quantum moment map at twice the critical level. We then go on to describe properties that emerge upon imposing (weak) graded unitarity. The vertex algebra $\CV$ underlying the relative semi-infinite cochain complex with coefficients in $\CV_{\rm matter}$ is formed by taking the $G$-invariant subalgebra of a product of $\CV_{\rm matter}$ and a symplectic fermion VOA ${\rm Sf}[\C^2 \otimes \fg]$ (see \emph{e.g.} Section \ref{sec:ex2}):
%%%%%%
\[
    \CV \coloneqq C(\wh{\fg}_{-2h^\vee}, \fg, \CV_{\rm matter}) \cong \left( \CV_{\rm matter} \otimes {\rm Sf}[\C^2 \otimes \fg] \right)^G~.
\]
%%%%%%
Our first result establishes that, under some mild but physically-natural conditions on the $\wh{\fg}_{-2h^\vee}$ action that we call ``good'' (see Def. \ref{dfn:goodaction}), the vertex algebra $\CV$ enjoys a rich algebraic structure analogous to that of differential forms on a compact K\"ahler manifold. 

%%%%%%
\textbf{Theorem} (Theorem \ref{thm:BRSTKahlerpackage})\textbf{.}
    \textit{Let $\CV_{\rm matter}$ be a weak graded-unitary vertex algebra equipped with a good Hamiltonian $\wh{\fg}_{-2h^\vee}$ action. Then $\CV$ is also a weak graded-unitary vertex algebra and is equipped with a pair of commuting differentials $\bQ^\pm$ satisfying
    %%%%%%
    \[
    	[\bQ^\alpha, \bQ^\beta] = 0~, \qquad [\overline{\bQ}_\alpha, \overline{\bQ}_\beta] = 0~, \qquad [\bQ^\alpha, \overline{\bQ}_\beta] = \delta^\alpha_\beta \Delta~,
    \]
    %%%%%%
    where $\overline{\bQ}_\pm \coloneqq (\bQ^\pm)^\dagger$ is the adjoint of $\bQ^\pm$ and $\Delta$ the corresponding ``Laplacian'', as well as a natural $\fsl(2)$ triple $\{\Pi, L, \Lambda\}$ with $\Pi^\dagger = \Pi$ and $L^\dagger = \Lambda$ such that $\bQ^\alpha$ transform in the standard representation thereof and $\overline{\bQ}_\alpha$ in its dual. Moreover, $\CV$ is equipped with a pair of compatible Hodge decompositions
    %%%%%%
    \[
    	\CV = \ker \Delta \oplus {\rm im} \bQ^+ \oplus {\rm im} \overline{\bQ}_+ = \ker \Delta \oplus {\rm im} \bQ^- \oplus {\rm im} \overline{\bQ}_-~.
    \]
    %%%%%%
    }
%%%%%%

The fact that $\CV$ admits a pair of commuting differentials rotated by an $\fsl(2)$ automorphism is not a new observation, see \emph{e.g.} Section 3.4.1 of \cite{Beem:2013sza}. When combined with (weak) graded unitarity, however, we see that $\CV$ bears a striking resemblance to the de Rham complex of a compact K\"{a}hler manifold.%
\footnote{It is worth noting that K\"{a}hler-like structures on relative semi-infinite cohomology also appear in work of Frenkel, Garland, and Zuckerman \cite{FrenkelGarlandZuckerman1986} in the context of genuinely unitary actions, inspired by results of Banks and Peskin \cite{BanksPeskin1986}; we will see that our results generalize them in the case of affine Lie algebras.} %

We turn to consequences of this theorem in Section \ref{sec:consequences}. The first result establishes that not only does the underlying chain complex admit a (weak) graded-unitary structure, so too does the cohomology.

%%%%%%
\noindent \textbf{Theorem} (Theorem \ref{thm:gradedunitaryBRST})\textbf{.}
    \textit{Relative semi-infinite $\wh{\fg}_{-2h^\vee}$-cohomology with coefficients in $\CV_{\rm matter}$
    %%%%%%
    \[
        H^{\scriptstyle{\frac{\infty}{2}}+\bullet}(\wh{\fg}_{-2h^\vee}, \fg, \CV_{\rm matter}) = H(\CV, \bQ^-)
    \]
    is also a weak graded-unitary vertex algebra. Moreover, if $\CV_{\rm matter}$ is a graded-unitary vertex algebra then so too is $H^{\scriptstyle{\frac{\infty}{2}}+\bullet}(\wh{\fg}_{-2h^\vee}, \fg, \CV_{\rm matter})$.}
%%%%%%

If we take $\CV_{\rm matter}$ to be a symplectic boson VOA, whose graded unitarity is sketched in Section 2.4.1 of \cite{ArabiArdehali:2025fad} and established in Proposition \ref{prop:Sbgradedunitarity}, this theorem implies that the VOAs coming from all Lagrangian $\CN=2$ superconformal gauge theories are graded unitary.

The remaining results in Section \ref{sec:consequences} capitalize on the analogy with K\"{a}hler geometry. In Proposition \ref{prop:automorphisms} we show that the cohomology of this complex is naturally equipped with a unitary action of ${\rm USp}(2)$ by vertex algebra automorphisms extending the ${\rm U}(1)$ action with graded subspaces $H^{\scriptstyle{\frac{\infty}{2}}+d}(\wh{\fg}_{-2h^\vee}, \fg, \CV_{\rm matter})$, reminiscent of the Lefschetz $\fsl(2)$ action on the cohomology of compact K\"{a}hler manifolds. Next, we transport results of Deligne, Griffiths, Morgan, and Sullivan \cite{DGMS} to establish a vertex-algebraic analogue of formality in Theorem \ref{thm:formality}, \emph{i.e.} we prove that the differential graded (DG) vertex algebras underlying relative semi-infinite cohomology of weak graded-unitary vertex algebras (with a good action) are quasi-isomorphic to their cohomology:

%%%%%%
\noindent \textbf{Theorem} (Theorem \ref{thm:formality})\textbf{.}
    \textit{The maps
    %%%%%%
    \[
        (\CV, \bQ^-) \overset{i}{\longleftarrow} (\ker \bQ^+, \bQ^-) \overset{\pi}{\longrightarrow} (H(\CV, \bQ^+), 0)
    \]
    %%%%%%
    where $i$ is the inclusion of the subcomplex $(\ker \bQ^+, \bQ^-)$ and $\pi$ is the projection onto $\bQ^+$-cohomology, are quasi-isomorphisms of DG vertex algebras.}
%%%%%%

\noindent Although the homotopy theory of vertex algebras is not yet well-developed, see, \emph{e.g.}, \cite{BD04, Caradot2023, Caradot2024} for some preliminary steps in this direction, we view this notion of formality as implying that the vertex-algebraic homotopy type of the DG vertex algebra $(\CV, \bQ^-)$ as being the same as its cohomology. In particular, there should be no Massey-like higher products on cohomology.

As another useful consequence of weak graded unitarity, we consider the case where the Lie algebra $\fg$ is a direct sum $\fg = \fg_1 \oplus \fg_2$; the differentials, their adjoints, and their Laplacians can then be written as a sum of two terms, one for $\fg_1$ and the other for $\fg_2$, and we show in Proposition \ref{prop:iteratedformality} that the total complex $(\CV, \bQ^-)$ is quasi-isomorphic to the iterated cohomology. An immediate corollary is that when $\fg$ the relative semi-infinite cohomology for $\fg$ is isomorphic to an iterated relative semi-infinite cohomology for its direct summands, \emph{cf.} Corollary \ref{cor:iteratedcohomology}.

%%%%%%
\textbf{Corollary} (Corollary \ref{cor:iteratedcohomology})\textbf{.}
    \textit{Let $\fg = \bigoplus_{k=1}^N \fg_k$ be a direct sum decomposition of $\fg$, then there is an isomorphism of vertex algebras
    %%%%%%
    \[
        H(\CV, \bQ^-) \simeq H(\dots H(\CV, \bQ^-_1), \dots,  \bQ^-_N{}^*)~.
    \]
    %%%%%%
    }
%%%%%%

It is worth noting that this result is expected from the perspective of four-dimensional physics, \emph{i.e.} for genuinely graded-unitary vertex algebras: the gauge couplings of superconformal gauge theories are exactly marginal and so it is inconsequential how one arrives at a given point the space of gauge couplings -- one could turn them on all at once, corresponding to the left-hand side, or one at a time, corresponding to the right-hand side, and arrive at the same result. The same cannot be said for the more general physical setups giving rise to weak graded-unitary vertex algebras, \emph{e.g.} in the examples coming from three-dimensional gauge theories where the gauge couplings are dimensionful and it is not guaranteed that one can move around in parameter space without consequence.

In Section \ref{sec:PVAandHLCR}, we describe other algebraic structures relevant to four-dimensional physics that can be extracted from knowing the graded-unitary vertex algebra $\CV$. We first consider the Poisson vertex algebra (PVA) ${\rm gr}_\mathfrak{R} \CV$ obtained from the associated graded with respect to the $\mathfrak{R}$-filtration of $\CV$ and establish in Lemma \ref{lem:HermitianPVA} that it inherits a positive-definite Hermitian inner product from that of $\CV$. Moreover, we prove an analogue of Theorem \ref{thm:BRSTKahlerpackage} in this context, see Proposition \ref{prop:PVAKahler} for more details. It is important to note that the inner product on ${\rm gr}_\mathfrak{R} \CV$ inherited from $\CV$ is not invariant with respect to its vertex Lie structure and so it is much less natural than that of $\CV$. Nonetheless, all of the consequences established in Section \ref{sec:consequences} for $\CV$ are similarly inherited by ${\rm gr}_\mathfrak{R} \CV$. For example, the associated graded DG PVAs of relative semi-infinite cochain complexes of graded-unitary vertex algebras with a good Hamiltonian $\wh{\fg}_{-2h^\vee}$ action are formal.

The last objects of interest are certain finite-type Poisson algebras that can be extracted from the Poisson vertex algebra ${\rm gr}_\mathfrak{R} \CV$ called the Hall--Littlewood chiral ring and denoted ${\rm HL}_{\CV}$.%
\footnote{There is additionally the Hall--Littlewood anti-chiral ring, but we do not focus on it as it is equivalent to the Hall--Littlewood chiral ring at the level of generality we consider.} %
We define these Poisson algebras using the additional $\Z$ grading by (twice) ${\rm U}(1)_r$ together with another modest addition to the above structure: this grading must satisfy a \emph{BPS bound} (see Definition \ref{dfn:HLOps}). When $\CV$ is the relative semi-infinite cochain complex with coefficients in a graded-unitary vertex algebra $\CV_{\rm matter}$, we relate the Hall--Littlewood chiral rings of $\CV$ and $\CV_{\rm matter}$ in Proposition \ref{prop:HLBRST}: it is given by derived Poisson reduction in the sense of \cite{Calaque2015, Safronov2017}.

%%%%%%
\noindent \textbf{Proposition} (Proposition \ref{prop:HLBRST})\textbf{.}
    \textit{The Hall--Littlewood chiral ring ${\rm HL}_\CV$ of $\CV$ can be identified with the derived Poisson reduction of ${\rm HL}_{\CV_{\rm matter}}$ by $G$. In other words,
    %%%%%%
    \[
    	{\rm Spec}~{\rm HL}_{\CV} \simeq ({\rm Spec}~{\rm HL}_{\CV_{\rm matter}})/\!\!/ G~.
    \]
    %%%%%%
    }
%%%%%%

\noindent Paired with the PVA analogue of Theorem \ref{thm:formality}, an immediate corollary of this result is that this derived Poisson reduction is formal, \emph{cf.} Corollary \ref{cor:HLformality}. For example, this establishes the formality of the derived Poisson reduction $\CR/\!\!/G$ for any complex symplectic representation $\CR$ of $G$ satisfying
%%%%%%
    \[
        \Tr_\CR(X^2) = 2 \Tr_{\rm ad}(X^2)
    \]
%%%%%%
for all $X \in \fg$. We are not aware of any similar results concerning the formality of derived Poisson reductions.

The results reported here leave open several interesting avenues to pursue, some of which we summarize here.
\begin{enumerate}
    \item We have established that the DG vertex algebras underlying relative semi-infinite cohomology of weak graded-unitary vertex algebras generally admit large symplectic outer automorphism groups. What is their physical interpretation?%
    \footnote{\label{footnote:outer_speculations}K. Costello has suggested to us that in the context of four-dimensional physics, this outer automorphism group could be a manifestation of $S$-duality: from the perspective of the holomorphic-topological twist, which can be $\Omega$-deformed to reproduce the VOAs \cite{Oh:2019bgz, Jeong:2019pzg}, the two gauginos exchanged by ${\rm USp}(2)$ descend to Wilson and 't Hooft lines. Alternatively, from a three-dimensional perspective, the diagonal torus of our ${\rm USp}(2)$ can be identified with the diagonal torus of the ${\rm SU}(2)_C$ symmetry, which suggests that our automorphism could be an avatar of this $R$-symmetry group but it is not entirely obvious that this needs to be the case.}%

    \item In some examples, multiple of the ${\rm USp}(2)$ automorphisms found in this note combine and enhance to a larger automorphism group. For example, the VOAs coming from a genus-$g$ class $\CS$ theory, \emph{e.g.} when realized as a suitable relative semi-infinite cohomology of the VOAs appearing in \cite{Arakawa:2018egx}, should witness an enhancement of ${\rm USp}(2)^{\otimes g}$ to ${\rm USp}(2g)$; see \cite{Beem:2021jnm} for the genus-2 class $\CS$ theory of type $A_1$, where two copies of ${\rm USp}(2)$ combine into a ${\rm USp}(4)$ automorphism group. How can this enhancement be understood systematically?
    
    \item The VOAs appearing in the SCFT/VOA correspondence are rather geometric in flavor and are often thought of as chiral quantizations of their parent SCFT's Higgs branch of vacua. These Higgs branches are naturally hyperk\"{a}hler cones and their twistor $\P^1$ of complex structure structures are rotated by the same ${\rm SU}(2)_R$ symmetry used to define the filtration in their graded-unitary structures. Can the graded-unitary structure on such a VOA be understood as arising from a kind of chiral quantization of the twistor $\P^1$ of the Higgs branch and can this be used to relate the hyperk\"{a}hler metric on the parent SCFT's Higgs branch \cite{Hitchin1987} to vertex algebraic data?
\end{enumerate}
    %!TEX Root = ../formality.tex

%---------------------------------------------------------%
\section{\label{sec:unitary}Graded-unitary vertex algebras}
%---------------------------------------------------------%

In this section, we review and refine the definition of a graded-unitary vertex algebra introduced recently in \cite{ArabiArdehali:2025fad}, and discuss the simple but relevant examples of symplectic boson and symplectic fermion VOAs. We will assume familiarity with the notion of a vertex (operator) algebra and will frequently omit the state-field correspondence $Y(-,z)$, instead writing, \emph{e.g.}, $x(z)$ rather than $Y(x,z)$. We denote the vacuum vector by $|0\rangle$, the translation operator by $\partial$, the conformal vector (when present) by $\omega$, and the modes of the stress tensor, \emph{i.e.}, the field associated to the conformal vector, by
%%%%%%
\begin{equation*}
    T(z) \coloneqq Y(\omega, z) = \sum_{n \in \Z} L_n z^{-n-2}~.
\end{equation*}
%%%%%%
Our conventions for mode indexing (for states other than the conformal vector) are such that
%%%%%%
\begin{equation*}
    x(z) \equiv Y(x, z) = \sum_{n \in \Z} x_{n} z^{-n-1}
\end{equation*}
%%%%%%
for all $x \in \CV$ (\emph{i.e.}, we use math conventions, albeit without parentheses on modes, rather than physics conventions). We \emph{do not} require (half\=/)integrality of conformal weights to be correlated with Grassmann parity.

%---------------------------------------------------------------------------------------%
\subsection{\label{subsec:mobius}M\"{o}bius vertex algebras and invariant bilinear forms}
%---------------------------------------------------------------------------------------%

Although the SCFT/VOA correspondence in four dimensions leads to graded-unitary VOAs, we will work in greater generality here. The largest class of vertex algebras we will consider relaxes the Virasoro symmetry of a VOA to an action of global conformal transformations.

%%%%%%
\begin{dfn}
    A \defterm{M\"{o}bius vertex algebra} is a $\frac{1}{2}\Z$-graded vertex algebra $\CV = \bigoplus_{h\in \scriptstyle{\frac{1}{2}}\Z} \CV_h$ with $\dim \CV_h < \infty$ and $\CV_h = 0$ for $h \ll 0$ equipped with endomorphisms $\{L_{\pm1}, L_0\}$ satisfying
    %%%%%%
    \begin{equation*}
        [L_m, L_n] = (m-n) L_{m+n}~,
    \end{equation*}
    %%%%%%
    such that
    %%%%%%
    \begin{enumerate}
        \item $L_m |0\rangle = 0$~,
        \item $L_{-1} = \partial$~,
        \item $L_0|_{\CV_h} = h ~ \id_{\CV_{h}}$~
        \item $[L_m, Y(x, z)] = \sum_{n=0}^{m+1} \binom{m+1}{n} z^{m+1-n} Y(L_{n-1} x, z)$
        %%%%%%
        for every $x \in \CV$.
    \end{enumerate}
    %%%%%%
    A vector $x \in \ker L_1$ is called a \defterm{quasi-primary} and a vector $y \in {\rm im}~L_{-1}$ is called a \defterm{descendant}.
\end{dfn}
%%%%%%

VOAs are trivially examples of M\"{o}bius vertex algebras, but there are M\"{o}bius vertex algebra that are not VOAs.

%%%%%%
\begin{physrmk}
     The (twisted-translated) algebra of Schur operators restricted to lie on half-BPS codimension-two defects in six-dimensional $(2,0)$ SCFTs are M\"{o}bius vertex algebras rather than VOAs \cite{Beem:2014kka}. Notable examples of M\"{o}bius vertex algebras that are expected to appear in this context are simple affine current algebras $L_k(\fg)$ at \emph{critical level} $k = -h^\vee$ and Drinfel'd--Sokolov reductions thereof. M\"{o}bius vertex algebras are sometimes instead called quasi-VOAs, \emph{cf.} \cite{FHL1993}.
\end{physrmk}
%%%%%%

M\"{o}bius vertex algebras admit an appropriate notion of invariant bilinear form; this definition is due to Frenkel, Huang, and Lepowski in the context of VOAs \cite{FHL1993}, see also \cite{Li1994}.\footnote{Certain choices of convention must be made in defining invariant bilinear forms in the presence of Grassmann-odd vectors and half-integer conformal weights. These choices are naturally related to a choice of definition for the contragredient representation of such a vertex algebra, see \cite{CKM24}.}

%%%%%%
\begin{dfn}
    Let $\CV$ be a M\"{o}bius vertex algebra. A bilinear form $(-,-)$ on $\CV$ is said to be \defterm{invariant} if
    %%%%%%
    \[
        (Y(x, z) u, v) = (-1)^{|x||u|} (u, Y(e^{z L_1} (e^{\ii \pi}z^{-2})^{L_0} x, z^{-1}) v)
    \]
    %%%%%%
    for all $u, v, x \in \CV$, where $|x|$, $|u|$ denote the Grassmann parity of $x$, $u$, and
    \[
        (L_m u, v) = (u, L_{-m} v)
    \]
    for all $m \in \{-1,0,1\}$.
\end{dfn}
%%%%%%

%%%%%%
\begin{rmk}
    The second condition implies that the pairing of vectors in $\CV_{h_1}$ and $\CV_{h_2}$ vanishes unless $h_1 = h_2$. A mild extension of \cite[Prop. 5.3.6]{FHL1993} to account for Grassmann-odd vectors and half-integral conformal weights implies that the restriction of $(-,-)$ to $\CV_h$ satisfies
    \[
        (x,y) = (-1)^{|x||y| + 2h} (y,x)
    \]
    for vectors $x$, $y$ that are homogeneous in Grassmann parity and conformal weight. If $\CV$ is a VOA, the first condition implies the second upon setting $x = \omega$.
\end{rmk}
%%%%%%

%%%%%%
\begin{rmk}
     A M\"{o}bius structure gives rise to a natural module structure on the restricted dual $M'$ of any grading-restricted generalized%
     \footnote{This condition requires that the module is graded $M = \bigoplus_{h \in \C} M_h$ such that
     %%%
     \begin{itemize}
         \item[1)] $\dim M_h < \infty$ for all $h \in \C$,
         \item[2)] $M_{h+n} = 0$ for all $h \in \C$ and $n \in \Z$ sufficiently negative,
         \item[3)] $L_0$ can be written as $L_0 = L_{0,S} + L_{0,N}$, where $L_{0,S}|_{M_h} = h~\id_{M_h}$ and $L_{0,N}$ is nilpotent.
     \end{itemize}
     %%%
     The vacuum module of a M\"{o}bius vertex algebra is, by the definition, a grading-restricted generalized module.} %
     module $M$ by the formula
    \[
        \langle Y(x, z) m', m\rangle = (-1)^{|x||m'|} \langle m', Y(e^{z L_1} (e^{\ii \pi} z^{-2})^{L_0} x, z^{-1}) m\rangle~,
    \]
    where $x \in \CV$, $m' \in M'$, $m \in M$, and $\langle-,-\rangle$ denotes the natural pairing of $M'$ and $M$. The resulting module is called the \emph{contragredient dual}. An invariant bilinear form on $\CV$ can equivalently be viewed as a $\CV$-module homomorphism $\CV \to \CV'$ given by $x \mapsto (x,-)$.
\end{rmk}
%%%%%%

We will need a strengthened version of M\"{o}bius vertex algebras. The following definitions are due to Dong, Li, Mason, and Montague \cite{DLMM} and Dong and Mason \cite{DongMason2004}, respectively, albeit in the context of VOAs.

%%%%%%
\begin{dfn}
    A M\"{o}bius vertex algebra is said to be \defterm{of CFT type} if the $\frac{1}{2}\Z$ grading satisfies $\CV_h = 0$ for $h < 0$ and $\CV_0 = \C |0\rangle$. It is said to be \defterm{of strong CFT type} if additionally $L_{1} \CV_1 = 0$.
\end{dfn}
%%%%%%

%%%%%%
\begin{physrmk}
    Vertex algebras that arise from four-dimensional $\CN=2$ SCFTs via the mechanism of \cite{Beem:2013sza} are always of strong CFT type. In brief, all Schur operators other than the identity operator have $h>0$, and those with $h = 1$ come from free vector multiplets or conserved currents and are necessarily quasi-primaries as a consequence of four-dimensional Ward identities.
\end{physrmk}
%%%%%%

%%%%%%
\begin{rmk}\label{rmk:strongpairing}
    Vertex algebras of strong CFT type come equipped with a unique-up-to-rescaling invariant bilinear form obtained from suitable two-point functions. Explicitly, (fixing normalization by imposing $(|0\rangle,|0\rangle)=1$) this form is given by%
    \footnote{If $x \in \CV_{h_x}$ is a quasi-primary, then this pairing takes the simplified form
    %%%%%%
    \[
        (x,y) \coloneqq \lim\limits_{z \to \infty} (e^{\ii \pi} z^2)^{h_x} \langle 0 | x(z) y(0) | 0 \rangle~,
    \]
    %%%%%%%
    which is perhaps a more familiar expression from two-dimensional CFT.}
    %
    %%%%%%
    \[
        (x, y) = {\rm Res}_{z}~ z^{-1} \langle 0| Y(e^{z L_1} (e^{\ii \pi} z^{-2})^{L_0}x,z^{-1}) y_{-1}|0\rangle~,
    \]
    %%%%%%
    where $\langle 0|$ is the linear functional on $\CV$ that annihilates $\CV_h$ for $h > 0$ and sends the vacuum vector $|0\rangle$ to $1$. That this defines an invariant bilinear form on $\CV$ follows from, \emph{e.g.}, \cite[Theorem 3.1]{Li1994} and that $\CV$ is of strong CFT type.
\end{rmk}
%%%%%%

Our interest will be in M\"{o}bius vertex algebras of CFT type for which there is a \emph{nondegenerate} invariant bilinear form. We establish that such vertex algebras are necessarily of strong CFT type and simple.

%%%%%%
\begin{prop}\label{prop:strong}
    Let $\CV$ be a M\"{o}bius vertex algebra of CFT type equipped with a nondegenerate invariant bilinear form. Then $\CV$ is of strong CFT type.
\end{prop}
%%%%%%

%%%%%%
\begin{proof}
    The bilinear form $(-,-)$ is equivalently a morphism from $\CV$ to its contragredient dual $\CV'$, as $V$-modules. Nondegeneracy of this pairing implies that this is an isomorphism, so a VOA $\CV$ equipped with a nondegenerate bilinear form is necessarily self-dual. Then Proposition 3.4 and the proof of Theorem 3.1 of \cite{Li1994}, which hold in the more general context of M\"{o}bius vertex algebras if one is only interested in the vacuum module, imply that there is an isomorphism of vector spaces
    %%%%%%
    \begin{equation*}
        {\rm Hom}(\CV, \CV') \simeq (\CV_0 / L_1 \CV_1)^*~.
    \end{equation*}
    %%%%%%
    The right-hand side is at most one-dimensional when $\CV$ is of CFT type, whereas self-duality implies that the left-hand side is at least one-dimensional. We can thus conclude both sides are one-dimensional and hence $L_1 \CV_1 = 0$.
\end{proof}
%%%%%%

%%%%%%
\begin{cor}\label{cor:simple}
    Let $\CV$ be a M\"{o}bius vertex algebra of CFT type equipped with a nondegenerate invariant bilinear form. Then $\CV$ is simple.
\end{cor}
%%%%%%

%%%%%%
\begin{proof}
    By Proposition \ref{prop:strong} we have that $\CV$ is of strong CFT type, and hence any bilinear form is given by (a multiple of) the one in Remark \ref{rmk:strongpairing}. To conclude that $\CV$ is simple we note that it cannot have any nontrivial ideals as nondegeneracy implies that any $x \in \CV_h$ can be mapped to the vacuum vector: there exists $y \in \CV_h$ such that $(x, y) \neq 0$ and so we have
    %%%%%%
    \[
        0 \neq (x, y) = (x, y_{-1} |0\rangle) = (-1)^{|x||y| + 2h} (y_{2h-1} x, |0\rangle)
    \]
    %%%%%%
    hence $y_{2h-1} x$ is a nonzero multiple of $|0\rangle$. As $|0\rangle$ is a cyclic vector for $\CV$, the result follows.
\end{proof}
%%%%%%

%---------------------------------------------------------------------------------%
\subsection{\label{subsec:good_filtrations}Nondegenerate half-integral filtrations}
%---------------------------------------------------------------------------------%

The next piece of structure we will need is (a mild generalization of) the notion of a good filtration as defined by Li in \cite{Li2004}.

%%%%%%
\begin{dfn}[Def. 2.1 of \cite{ArabiArdehali:2025fad}]
    A \defterm{half-integral filtration} of a $\Z_2$-graded vertex algebra $\CV = \CV^0 \oplus \CV^{\frac{1}{2}}$ is an increasing, exhaustive filtration
    %%%%%%
    \begin{equation*}
        \dots \subseteq \CF_{-1+\alpha} \CV^\alpha \subseteq \CF_{0 +\alpha} \CV^\alpha \subseteq \CF_{1+\alpha} \CV^\alpha \subseteq \dots~, \qquad \CV^\alpha = \bigcup_{p\in\Z} \CF_{p + \alpha} \CV^\alpha~, \qquad \alpha = 0, \tfrac{1}{2}~,
    \end{equation*}
    %%%%%%
    such that the vacuum vector belongs to $\CF_0 \CV^0$ and
    %%%%%%
    \[
        x_n y \in \CF_{p + q + \alpha + \beta} \CV^{\alpha + \beta}~
    \]
    %%%%%%
    for every $x \in \CF_{p + \alpha} \CV^\alpha$, $y \in \CF_{q + \beta} \CV^\beta$, and $n \in \Z$. A half-integral filtration is called \defterm{good} if additionally
    %%%%%%
    \[
        x_n y \in \CF_{p + q + \alpha + \beta - 1} \CV^{\alpha + \beta}~,\qquad n\geqslant0~.
    \]
    %%%%%%
    A \defterm{half-integer-filtered vertex algebra} is a $\mathbb{Z}_2$-graded vertex algebra equipped with a half-integral good filtration.
\end{dfn}
%%%%%%

%%%%%%
\begin{rmk}
    Note that $\mathcal{V}^0$ is itself a filtered vertex algebra in the sense of \cite{Li2004}. We will often simplify the notation and write the filtered components as $\CF_p \CV$ for $p \in \frac{1}{2}\Z$ with the understanding that the elements of $\CF_p \CV$ with $p \in \Z$ (resp. $p \in \frac{1}{2} + \Z$) belong to $\CV^0$ (resp. $\CV^{\scriptstyle{\frac{1}{2}}}$). We denote by $s$ the $\Z_2$ parity automorphism, so $s$ acts as the identity on $\CV^0$ and as $-1$ on $\CV^{\scriptstyle{\frac{1}{2}}}$. Obviously, $s$ preserves the filtration, $s(\CF_p \CV) = \CF_p\CV$.
\end{rmk}
%%%%%%

The following is a simple generalization of Proposition 4.1 of \cite{Li2004}, \emph{cf.} Section 2.2 of \cite{ArabiArdehali:2025fad}.

%%%%%%
\begin{prop}\label{prop:assocgradedPVA}
    Let $(\CV,\CF_\bullet)$ be a half-integer-filtered vertex algebra. Then the associated graded vector space
    %%%%%%
    \[
        {\rm gr}_\CF \CV = \bigoplus_{p \in \scriptstyle{\frac{1}{2}}\Z} \CF_{p} \CV / \CF_{p-1} \CV
    \]
    %%%%%%
    is naturally a $\frac{1}{2}\Z$-graded Poisson vertex algebra.
\end{prop}
%%%%%%

The vertex algebras in which we will ultimately be interested come equipped with a nondegenerate bilinear form that is compatible with the half-integral filtration in the following sense.

%%%%%%
\begin{dfn}
    Let $\CV$ be a M\"{o}bius vertex algebra equipped with a nondegenerate invariant bilinear form $(-,-)$. A (half-integral) filtration $\CF_\bullet$ of $(\CV, (-,-))$ is called \defterm{nondegenerate} if the bilinear form $(-,-)$ is nondegenerate when restricted to any filtered component.
\end{dfn}
%%%%%%

If $\CF_\bullet$ is a nondegenerate half-integral filtration, then the filtration on the underlying vector space can be canonically refined to a grading by imposing orthogonality between different graded subspaces,
%%%%%%
\begin{equation*}
    \CV = \bigoplus_{\substack{h \in \scriptstyle{\frac{1}{2}}\N \\ p \in \scriptstyle{\frac{1}{2}}\Z}} \CV_{h,p}~.
\end{equation*}
%%%%%%
The filtration is recovered as usual from the grading,
%%%%%%
\begin{equation*}
        \CV^\alpha = \bigoplus_{\substack{h \in \scriptstyle{\frac{1}{2}}\N \\ p \in \Z}} \CV_{h,p + \alpha}~, \qquad \CF_{p + \alpha} \CV^\alpha = \bigoplus_{\substack{h \in \scriptstyle{\frac{1}{2}}\N\\ p' \leqslant p}} \CV_{h, p' + \alpha}~,
\end{equation*}
%%%%%%
where $p, p' \in \Z$ and $\alpha = 0, \frac{1}{2}$. With this grading on $\CV$, we can define an order-four linear map $\sigma$ by
%%%%%%
\begin{equation*}
    \sigma|_{\CV_{h,p}} = \ii^{2p}~\id_{\CV_{h,p}}~,
\end{equation*}
%%%%%%
so $\sigma\circ\sigma=s$. The bilinear form $(-,-)$ restricts to a nondegenerate bilinear form on each graded component $\CV_{h,p}$. Note that the pairing $(x,y)$ vanishes unless $x, y$ both belong to the same graded component, \emph{i.e.}, $h_x = h_y$ and $p_x = p_y$, so in particular, we have
%%%%%%
\begin{equation*}
        (\sigma(x), y) = (x, \sigma(y))~.
\end{equation*}
%%%%%%
Moreover, it is straightforward to see that the action of M\"{o}bius transformations preserves graded subspaces,
%%%%%%
\[
    L_{m} : \CV_{h,p} \to \CV_{h-m,p}~.
\]
%%%%%%

We now prove a simple result regarding nondegenerate half-integral filtrations that establishes a shortening condition that limits how much a field can change the corresponding grading.

%%%%%%
\begin{prop}\label{prop:shortening}
    Let $\CV$ be a M\"{o}bius vertex algebra of CFT type equipped with a nondegenerate invariant bilinear form $(-,-)$. Let $\CF_\bullet$ be a nondegenerate half-integral filtration of $(\CV, (-,-))$ and
    %%%%%%
    \begin{equation*}
        \CV = \bigoplus_{\substack{h \in \scriptstyle{\frac{1}{2}} \N\\ p \in \scriptstyle{\frac{1}{2}}\Z}} \CV_{h,p}~,
    \end{equation*}
    %%%%%%
    the corresponding grading on the underlying vector space. Choose $x \in \CV_{h_x,p}$ and $n \in \Z$ and denote by
    %%%%%%
    \begin{equation*}
        x_n = \sum_{\substack{p' \in \scriptstyle{\frac{1}{2}}\Z\\ p' \leqslant p}} x^{[p']}_n~,
    \end{equation*}
    %%%%%%
    the decomposition of $x_n$ into homogeneous components with respect to the grading, \emph{i.e.},
    %%%%%%
    \[
    x_n^{[p']}:\CV_{h,q}\to\CV_{h+h_x-n-1,q+p'}~.
    \]
    %%%%%%
    Then the component $x^{[p']}_n$ vanishes unless $p' \equiv p \mod \Z$. Moreover, $x^{[p']}_n = 0$ if $p' < -p$, so $x_n$ can only increase or decrease degree by at most $p$.
\end{prop}
%%%%%%

%%%%%%
\begin{proof}
    The first assertion follows from the fact that the OPE is compatible with the $\Z_2$ grading, with respect to which $x$ has weight $(-1)^{2p}$.

    For the second assertion, choose $x \in \CV_{h_x, p}$. The invariance $(-,-)$ implies that the adjoint of the mode $x_n^{[p']}$ is 
    %%%%%%
    \[
        (x_n^{[p']})^* = e^{\ii \pi h_x}\sum_{k \geqslant 0} \frac{1}{k!} (L_1^k x)_{2h_x -n - 2 - k}^{[-p']}~,
    \]
    %%%%%%
    which is necessarily a finite sum. As $L_1$ preserves $\CF$-degree it follows that $L_1^k x \in \CV_{h_x-k, p}$ and hence the right-hand side vanishes if $-p' > p$ as $\CF_\bullet$ is a filtration. Nondegeneracy of $(-,-)$ implies $x_n^{[p']} = 0$ if $p' < p$, as desired.
\end{proof}
%%%%%%

An immediate corollary is that this shortening also holds for elements in a given filtered piece.

%%%%%%
\begin{cor}
    Let $\CV$ be a M\"{o}bius vertex algebra of CFT type equipped with a nondegenerate invariant bilinear form $(-,-)$. Let $\CF_\bullet$ be a nondegenerate half-integral filtration of $(\CV, (-,-))$ and
    %%%%%%
    \begin{equation*}
        \CV = \bigoplus_{\substack{h \in \scriptstyle{\frac{1}{2}} \N\\ p \in \scriptstyle{\frac{1}{2}}\Z}} \CV_{h,p}~,
    \end{equation*}
    %%%%%%
    the corresponding grading on the underlying vector space. Choose $x \in \CF_p\CV_{h_x}$ and $n \in \Z$ and denote by
    %%%%%%
    \begin{equation*}
        x_n = \sum_{\substack{p' \in \scriptstyle{\frac{1}{2}}\Z\\ p' \leqslant p}} x^{[p']}_n~,
    \end{equation*}
    %%%%%%
    the decomposition of $x_n$ into homogeneous components with respect to the grading. Then the component $x^{[p']}_n$ vanishes unless $p' \equiv p \mod \Z$. Moreover, $x^{[p']}_n = 0$ if $p' < -p$, \emph{i.e.}, $x_n$ can only increase or decrease the degree by at most $p$.
\end{cor}
%%%%%%

Another corollary of this shortening condition is that the filtration must be truncated.

%%%%%%
\begin{cor}
    Let $\CV$ be a M\"{o}bius vertex algebra of CFT type equipped with a nondegenerate invariant bilinear form $(-,-)$ and let $\CF_\bullet$ be a nondegenerate half-integral filtration of $(\CV, (-,-))$. Then the graded components satisfy
    \[
        \CV_{h,p} = 0~, \qquad p < 0~.
    \]
    In particular, $\CF_\bullet$ is necessarily truncated, \emph{i.e.}, $\CF_p \CV = 0$ for $p < 0$.
\end{cor}
%%%%%%

%%%%%%
\begin{proof}
    Proposition \ref{prop:shortening} implies that the field corresponding to any vector $x \in \CV_{h,p}$ for $p < 0$ must vanish. This is because $x^{[p']}_n = 0$ for all $p' \geqslant 0 (> p)$ as $\CF_\bullet$ is a filtration, whereas Proposition \ref{prop:shortening} implies that it must also vanish for all $p' \leqslant 0 (< -p)$. It follows that $x_n = 0$ for every $n$ and hence $x = x_{-1} |0\rangle = 0$.
\end{proof}
%%%%%%

%%%%%%
\begin{physrmk}
    This shortening condition is a vertex-algebraic cousin of the shortening condition on star products that arise in the context of three-dimensional $\CN=4$ superconformal field theories \cite{Beem:2016cbd, Etingof:2019guc}. Physically, as in the star-product case, this result captures the fact that the extra grading associated with the filtration is interpreted to be an ${\rm SU}(2)_R$ charge, and the twisted translates of Schur operators with charge $R$ are sums of operators with charges ranging from $-R$ to $R$. The truncation condition realizes the four-dimensional fact that Schur operators necessarily have $R \geqslant 0$.
\end{physrmk}
%%%%%%

%----------------------------------------------------------------------------------%
\subsection{\label{subsec:quaternionic}Quaternionic structures and graded unitarity}
%----------------------------------------------------------------------------------%

Recall that if $x,y$ are states in a M\"{o}bius vertex algebra with conformal weight $h$, the pairing $(x,y)$ satisfies
%%%%%%
\begin{equation*}
    (x, y) = (-1)^{|x||y| + 2h} (y,x)~.
\end{equation*}
%%%%%%
In defining notions of unitarity one expects a Hermitian form, which necessitates introducing additional structure with which to modify the bilinear form. Due to the physical similarity between graded-unitary VOAs and Hermitian short star-products coming from the deformation quantization of hyperk\"{a}hler cones---the former \cite{Beem:2013sza, ArabiArdehali:2025fad} being a four-dimensional analogue of the latter \cite{Beem:2016cbd}---we adopt similar nomenclature and notation to \cite{Etingof:2019guc}.

%%%%%%
\begin{dfn}
    A \defterm{conjugation} of a vertex algebra $\CV$ is an anti-linear vector space isomorphism $\rho: \CV \to \CV$ that commutes with the translation operator
    %%%%%%
    \[
        \rho \circ \partial = \partial \circ \rho~,
    \]
    %%%%%%
    preserves the vacuum vector
    %%%%%%
    \[
        \rho(|0\rangle) = |0\rangle~,
    \]
    %%%%%%
    and satisfies
    %%%%%%
    \[
        \rho(x_n y) = (-1)^{|x||y|} \rho(x)_n \rho(y)
    \]
    %%%%%%
    for every homogeneous $x, y \in \CV$ and $n \in \Z$. In the M\"{o}bius case, $\rho$ is required to commute with $L_{\pm1,0}$
    %%%%%%
    \[
        L_m \circ \rho = \rho \circ L_m~, \quad m \in \{-1,0,1\}~.
    \]
    %%%%%%
    In the VOA case, $\rho$ is required to preserve the conformal vector,
    %%%%%%
    \[
        \rho(\omega) = \omega~.
    \]
    %%%%%%
    If $\CF_\bullet$ is a half-integral filtration on $\CV$, we require that $\rho$ preserves each filtered component
    %%%%%%
    \[
        \rho(\CF_p \CV) \subseteq \CF_p \CV~.
    \]
    %%%%%%
\end{dfn}
%%%%%%

%%%%%%
\begin{rmk}
    If $\CV$ is a M\"{o}bius vertex algebra of CFT type, it follows that $\rho(|0\rangle)$ is a non-zero multiple of $|0\rangle$ and hence $\rho$ preserves the vacuum vector automatically as $|0\rangle_{-1} = \id_\CV$. When $\CV$ is a VOA of CFT type, it suffices to require that a conjugation satisfies
    %%%%%%
    \[
        \rho(x_n y) = (-1)^{|x||y|} \rho(x)_n \rho(y)~,
    \]
    %%%%%%
    and preserves the conformal vector. Indeed, as $\partial = L_{-1}$, these conditions imply that $\rho$ commutes with $\partial$ and, moreover, ensures that it preserves conformal weight spaces. Finally, the filtration-preserving property of $\rho$, along with the definition of half-integral filtration, implies that $\rho$ commutes with $s$.
\end{rmk}
%%%%%%

The following is a straightforward application of the skew-symmetry property of a vertex algebra and the definition of a conjugation.

%%%%%%
\begin{lem}
    Let $\CV$ be a M\"{o}bius vertex algebra equipped with a (nondegenerate) invariant bilinear form $(-,-)$ and let $\rho$ be a conjugation of $\CV$. Then
    %%%%%%
    \[
        (\rho(x), \rho(y)) = \overline{(y,x)}~.
    \]
    %%%%%%
\end{lem}
%%%%%%

With this lemma in mind, we make the following definition.

%%%%%%
\begin{dfn}
    Let $\CV$ be a M\"{o}bius vertex algebra equipped with a nondegenerate invariant bilinear form $(-,-)$ and let $\CF_\bullet$ be a nondegenerate half-integral filtration of $(\CV, (-,-))$. A \defterm{quaternionic structure} on $(\CV, (-,-), \CF_\bullet)$ is a conjugation $\rho$ further satisfying
    %%%%%%
    \[
        \rho \circ \rho = s ~, \qquad \rho \circ \sigma = \sigma^{-1} \circ \rho~.
    \]
    %%%%%%
\end{dfn}
%%%%%%

%%%%%%
\begin{rmk}
    By the filtration-preserving property of $\rho$, the condition $\rho\circ\rho=s$ implies that $\rho$ preserves graded subspaces $\CV_{h,p}$. It follows that $\rho\circ\sigma=\sigma^{-1}\circ\rho$ automatically.
\end{rmk}
%%%%%%

%%%%%%
\begin{cor}\label{cor:Hermitianform}
    Let $\CV$ be a M\"{o}bius vertex algebra equipped with a nondegenerate invariant bilinear form $(-,-)$, $\CF_\bullet$ a nondegenerate half-integral filtration of $(\CV, (-,-))$, and $\rho$ a quaternionic structure on $(\CV, (-,-), \CF_\bullet)$. The sesquilinear form defined by
    %%%%%%
    \[
        \langle x | y \rangle \coloneqq ((\sigma \circ \rho)(x), y)~,
    \]
    %%%%%%
    is Hermitian,
    %%%%%%
    \[
        \langle x | y \rangle = \overline{\langle y|x\rangle}~.
    \]
    %%%%%%
\end{cor}
%%%%%%

%%%%%%
\begin{rmk}
    Note that we use the convention from physics that a sesquilinear form is linear in the second argument and anti-linear in the first.
\end{rmk}
%%%%%%

%%%%%%
\begin{proof}
    We compute:
    %%%%%%
    \[
    \begin{aligned}
        \langle x | y \rangle & = ((\sigma \circ \rho)(x), y) = (\rho(x), \sigma(y)) = \overline{((\rho^{-1}\circ \sigma)(y), x)}\\
        & = \overline{((\sigma \circ \rho)(y),x)} = \overline{\langle y | x \rangle} ~.
    \end{aligned}
    \]
    %%%%%%
\end{proof}
%%%%%%

We are now in a position to define graded-unitary vertex algebras (\emph{cf.} Def. 2.2 of \cite{ArabiArdehali:2025fad}), as well as a slightly weakened version thereof.

%%%%%%
\begin{dfn}[]
    A \defterm{weak graded-unitary vertex algebra} is a quadruple ($\CV$, $(-,-)$, $\mathfrak{R}_\bullet$, $\rho$), where
    %%%%%%
    \begin{enumerate}
        \item[(i)] $\CV = \bigoplus_{h, \alpha} \CV^\alpha_{h}$ is a M\"{o}bius vertex algebra with additional $\Z_2$ grading and associated $\Z_2$ parity operation $s$,
        \item[(ii)] $(-,-)$ is a nondegenerate invariant bilinear form on $\CV$,
        \item[(iii)] $\mathfrak{R}_\bullet$ is a nondegenerate half-integral filtration of $(\CV, (-,-))$ with respect to $s$,
        \item[(iv)] and $\rho$ is a quaternionic structure on $(\CV, (-,-), \mathfrak{R}_\bullet)$,
    \end{enumerate}
    %%%%%%
    such that the Hermitian form defined in Corollary \ref{cor:Hermitianform} is positive-definite. We say that this quadruple is a \defterm{graded-unitary vertex algebra} if, in addition, $\mathfrak{R}_\bullet$ is good.
\end{dfn}
%%%%%%

%%%%%%
\begin{rmk}
    Following \cite{ArabiArdehali:2025fad}, for a weak graded-unitary vertex algebras we denote the nondegenerate half-integral filtration by $\mathfrak{R}_\bullet$ and refer to it as the $\mathfrak{R}$-filtration due to its four-dimensional origin of the filtration in the SCFT/VOA correspondence. When refining to a grading, we use $R$ rather than $p$.
\end{rmk}
%%%%%%

%%%%%%
\begin{physrmk}
    Corollary \ref{cor:simple} implies that a weak graded-unitary vertex algebra of CFT type is necessarily simple. In particular, all VOAs appearing in the SCFT/VOA correspondence are simple. This is a result that is well known to experts, but we do not believe an explicit proof has appeared in the literature.
\end{physrmk}
%%%%%%

An immediate consequence of weak graded unitarity is an analogue of the spin-statistics relation, which relates the Grassmann parity to spin (half-integrality of conformal weight) and $\mathfrak{R}$-parity.

%%%%%%
\begin{prop}[Graded-Unitary Spin-Statistics]\label{prop:spinstatistics}
    Let $(\CV$, $(-,-)$, $\mathfrak{R}_\bullet$, $\rho)$ be a weak graded-unitary vertex algebra. Then the Grassmann parity of any nonzero element $x \in \CV_{h, R}$ satisfies
    %%%%%%
    \[
        (-1)^{|x|} = (-1)^{2(h+R)}~.
    \]
    %%%%%%
\end{prop}
%%%%%%

%%%%%%
\begin{proof}
    Choose $x \in \CV_{h, R}$ nonzero. We compute:
    %%%%%%
    \[
    \begin{aligned}
        \langle \rho(x) | \rho(x) \rangle & = \ii^{2R} (\rho^2(x), \rho(x))\\
        & = \ii^{2R} (-1)^{2R} (x, \rho(x))\\
        & = \ii^{2R} (-1)^{2R} (-1)^{|x||x| + 2h} (\rho(x), x)\\
        & = (-1)^{2R + |x| + 2h} \langle x | x \rangle~.
    \end{aligned}
    \]
    %%%%%%
    As $\langle - | - \rangle$ is positive-definite, both $\langle \rho(x) | \rho(x) \rangle$ and $\langle x | x\rangle$ are positive and hence
    %%%%%%
    \[
        1 = (-1)^{2R + |x| + 2h}~,
    \]
    %%%%%%
    as desired.
\end{proof}
%%%%%%

%%%%%%
\begin{rmk}
    The vector space underlying a weak graded-unitary vertex algebra is naturally a $\frac12\N\times\frac12\N$-graded complex Hilbert space,
    %%%%%%
    \[
        \CV = \bigoplus_{\substack{h \in \scriptstyle{\frac{1}{2}} \N\\ R \in \scriptstyle{\frac{1}{2}}\N}} \CV_{h, R}~,
    \]
    %%%%%%
    with finite-dimensional graded components such that the inner product $\langle x | y \rangle$ of two homogeneous elements vanishes if $(h_x, R_x) \neq (h_y, R_y)$. We note that for a quasi-primary $x \in \CV$ of conformal weight $h_x$, the adjoint of $x_n$ with respect to $\langle-|-\rangle$ is given by
    %%%%%%
    \[
		(x_n)^\dagger = e^{\ii \pi h_x} \sigma^{-1} \circ \rho \circ x_{2 h_x - 2 - n} \circ \rho^{-1} \circ \sigma~.
    \]
    %%%%%%
    This can be seen by the following computation:
    %%%%%%
    \[
    \begin{aligned}
        \langle x(z) u | v\rangle & = ((\sigma \circ \rho)(x(z) u), v) = (\rho(x)(z) \rho (u), \sigma (v))\\
        & = (e^{\ii \pi} z^{-2})^{h_x} (\rho(u), \rho(x)(z^{-1}) \sigma(v))\\
        & = (e^{\ii \pi} z^{-2})^{h_x} \langle u | (\sigma^{-1} \circ \rho) \circ x(z^{-1}) \circ (\rho^{-1} \circ \sigma) v\rangle~.
    \end{aligned}
    \]
    %%%%%%
    where $u, v$ are arbitrary vectors in $\CV$. The map
    %%%%%%
    \[
        x(z) \to (e^{\ii \pi} z^{-2})^{h_x}(\sigma^{-1} \circ \rho) \circ x(z^{-1}) \circ (\rho^{-1} \circ \sigma)
    \]
    %%%%%%
    captures the physical Hermitian conjugation in radial quantization.
\end{rmk}
%%%%%%

%%%%%%
\begin{rmk}\label{rmk:Zgrading}
    Many weak graded-unitary vertex algebras of interest come equipped with an internal $\Z$ grading by fermion number that refines Grassmann parity. The vector space underlying such a weak graded-unitary vertex algebra is then a $\frac12\N\times\frac12\N\times\Z$-graded complex Hilbert space,
    %%%%%%
    \[
        \CV = \bigoplus_{\substack{h \in \scriptstyle{\frac{1}{2}} \N\\ R \in \scriptstyle{\frac{1}{2}}\N\\ d\in \Z}} \CV_{h, R, d}~,
    \]
    %%%%%%
    with the pairing $\langle x | y\rangle$ vanishing unless $(h_x, R_x, d_x) = (h_y, R_y, -d_y)$. In situations arising from physics, there may be other constraints that should be imposed on the internal $\Z$ grading, see \emph{e.g.} Definition \ref{dfn:HLOps}.
\end{rmk}
%%%%%%

%%%%%%
\begin{physrmk}\label{physrmk:Zgrading}
    In the context of vertex algebras coming from the SCFT/VOA correspondence, the $\Z$ grading refining Grassmann parity can also be identified with (twice) the ${\rm U}(1)_r$ charge $d = 2r$. In light of Graded-Unitary Spin-Statistics we see that the $\Z_2$ parity operator can then be expressed in terms of $d$ and $h$ according to $s|_{\CV_{h,d}} = (-1)^{2h+d} \id_{\CV_{h,d}}$, mirroring the fact that $R \equiv h + r \mod \Z$ for Schur operators, \emph{cf.} Section 2.2 of \cite{ArabiArdehali:2025fad}. An additional constraint comes from a BPS bound of the four-dimensional SCFT, which requires that $R + |r| \leqslant h$.
\end{physrmk}
%%%%%%

%---------------------------------------%
\subsection{\label{sec:examples}Examples}
%---------------------------------------%

In this subsection we present several important examples, starting with the observation that a unitary structure on a M\"{o}bius vertex algebra, as defined by, \emph{e.g.}, \cite{Kac:2020ytv} (see also \cite{DongLin2014, AiLin2017, CapriKawahigashiLongoWeiner2018, RaymondTanimotoTener2022} for related work) gives rise to a weak graded-unitary structure. We then present two examples of genuinely graded-unitary VOAs that arise from free four-dimensional $\CN=2$ SCFTs, \emph{cf.} Section 2.4 of \cite{ArabiArdehali:2025fad}: the symplectic boson VOA, coming from the theory of a free hypermultiplet in four dimensions, and the symplectic fermion VOA, coming from the theory of a free vector multiplet in four dimensions. As we will see in Section \ref{sec:Kahler}, the graded-unitary structure on the symplectic fermion VOA allows us to define a (weak) graded-unitary structure on the DG vertex algebra underlying the relative semi-infinite cohomology of a (weak) graded-unitary vertex algebra when the affine currents satisfy relatively mild and physically-motivated conditions.

%%%%%%%%%%%%%%%%%%%%%%%%%%%%%%%%%%%%%%%%%%%%%%%%%%%%%%%%%%%%%%%%%%%%%%%%%%%%
\subsubsection{\label{sec:ex0}Example 0: unitary M\"{o}bius vertex algebras}
%%%%%%%%%%%%%%%%%%%%%%%%%%%%%%%%%%%%%%%%%%%%%%%%%%%%%%%%%%%%%%%%%%%%%%%%%%%%

A unitary structure, as defined by \cite{Kac:2020ytv} in the case of VOAs and adapted to the M\"{o}bius case in line with \cite{CapriKawahigashiLongoWeiner2018, RaymondTanimotoTener2022}, starts with a suitably defined invariant bilinear form with respect to an anti-linear automorphism $\phi$.

%%%%%%
\begin{dfn}[Def. 4.1 of \cite{Kac:2020ytv}]
    Let $\CV$ be a M\"{o}bius vertex algebra. An \defterm{anti-linear automorphism} of $\CV$ is an anti-linear map $\phi: \CV \to \CV$ such that
    %%%%%%%
    \begin{itemize}
        \item[(i)] $\phi(x_n y) = \phi(x)_n \phi(y)$ for all $x, y \in \CV$ and $n \in \Z$,
        \item[(ii)] $\phi \circ L_m = L_m \circ \phi$ for all $m \in \{-1, 0, 1\}$.
    \end{itemize}
    %%%%%%
    Let $\phi:\CV \to \CV$ be an anti-linear involution, \emph{i.e.} an anti-linear automorphism with $\phi^2 = \id_\CV$. A Hermitian form $\langle-,-\rangle$ on $\CV$ is called $\phi$-\defterm{invariant} if
    %%%%%%
    \[
        \langle u, Y(x, z) v \rangle = \langle e^{-\scriptstyle{\frac{\ii \pi |x|}{2}}}Y( e^{z L_1} (e^{-\ii \pi}z^{-2})^{L_0} \phi(x), z^{-1}), v\rangle
    \]
    %%%%%%
    for all $u, v, x \in \CV$.
\end{dfn}
%%%%%%

%%%%%%
\begin{rmk}
    It is important to note that the bilinear form $(-,-)$ defined by
    %%%%%%
    \[
        (u, v) \coloneqq \begin{cases}
            \langle \phi(u), v \rangle & |u| = 0\\
            -\ii ~ \langle \phi(u), v \rangle & |u| = 1
        \end{cases} \qquad \Longleftrightarrow \qquad \langle u, v\rangle = \begin{cases}
            (\phi(u), v) & |u| = 0\\
            \ii ~ (\phi(u), v) & |u| = 1
        \end{cases}
    \]
    %%%%%%
    is invariant in the sense of Section \ref{subsec:mobius}. That the factor of $e^{-\scriptstyle{\frac{\ii \pi |x|}{2}}}$ in the definition of $\phi$-invariance becomes the factor of $(-1)^{|x||u|}$ in the ordinary definition of invariance is related to the translation between various definitions of the contragredient action (see, \emph{e.g.}, Remark 3.5 of \cite{CKM24} for a brief discussion of this translation).
\end{rmk}
%%%%%%

%%%%%%
\begin{dfn}[Def. 4.8 of \cite{Kac:2020ytv}]
    A \defterm{unitary M\"{o}bius vertex algebra} is a triple $(\CV, \phi, \langle -, -\rangle)$, where
    \begin{itemize}
        \item[(i)] $\CV$ is a M\"{o}bius vertex algebra $\CV$,
        \item[(ii)] $\phi$ is an anti-linear involution of $\CV$,
        \item[(iii)] and $\langle -,- \rangle$ is a $\phi$-invariant Hermitian form on $\CV$,
    \end{itemize}
    such that $\langle -, - \rangle$ is positive-definite.
\end{dfn}
%%%%%%

Such a M\"{o}bius vertex algebra can be given the structure of a weak graded-unitary vertex algebra. As stated above, there is a natural nondegenerate invariant bilinear form $(-,-)$ arising from the Hermitian form in the unitary structure, and we take this to be the nondegenerate invariant bilinear form underlying the weak graded-unitary structure. We take the $\Z_2$ grading to be trivial, \emph{i.e.} $\CV^0 = \CV$ and $\CV^{\scriptstyle{\frac{1}{2}}} = 0$, and similarly for the half-integral filtration
%%%%%%
\[
    \mathfrak{R}_R \CV = \begin{cases}
        0 & R < 0 ~ \text{ or } ~ R \equiv \tfrac{1}{2} \mod \Z~,\\
        \CV & R \in \N~,
    \end{cases}
\]
%%%%%%
which is certainly not a good (half-integral) filtration but is nonetheless nondegenerate. Refining this filtration to a grading leaves $\CV$ concentrated entirely in $\mathfrak{R}$-degree $0$, and $\sigma$ is simply the identity map.

We are thus left with defining a quaternionic structure and checking that the bilinear form $\langle - | - \rangle$ as defined in Corollary \ref{cor:Hermitianform} is positive-definite. To do this, we note that we can turn an anti-linear automorphism into a conjugation by composing with a square-root of the Grassmann parity operator, so the natural candidate for the quaternionic structure is
%%%%%%
\[
    \rho(x) \coloneqq \begin{cases}
        \phi(x) & |x| = 0~,\\
        \ii~\phi(x) & |x| = 1~.
    \end{cases}
\]
%%%%%%
That this conjugation defines a quaternionic structure is immediate from the fact that both $s$ and $\sigma$ are trivial. Finally, positive definiteness of the Hermitian form $\langle -|- \rangle$ follows from the positive definiteness of the Hermitian form $\langle - , -\rangle$ as they are in fact identical,
%%%%%
\[
    \langle u | v \rangle \coloneqq ((\sigma\circ \rho)(u), v) = (\rho(u), v) = \begin{cases}
        (\phi(u), v) & |u| = 0\\
        \ii~(\phi(u), v) & |u| = 1
    \end{cases} \bigg\} = \langle u, v \rangle~.
\]
%%%%%

In light of Proposition \ref{prop:spinstatistics}, we recover the fact that unitary M\"{o}bius vertex algebras obey the usual spin-statistics relation.

%%%%%%%%%%%%%%%%%%%%%%%%%%%%%%%%%%%%%%%%%%%%%%%%%%%%%%%%%%%
\subsubsection{\label{sec:ex1}Example 1: symplectic bosons}
%%%%%%%%%%%%%%%%%%%%%%%%%%%%%%%%%%%%%%%%%%%%%%%%%%%%%%%%%%%

We now consider an example of a genuinely graded-unitary vertex algebra: the symplectic boson VOA. Choose a symplectic vector space $V$ over $\C$. Let $\{e_a\}$ be a basis of $V$ and let $\Omega(-,-)$ denote the symplectic form on $V$; we set $\Omega_{ab} = \Omega(e_a, e_b)$ and denote the inverse of this matrix by $\Omega^{ab}$
%%%%%%
\[
    \Omega_{ab} \Omega^{bc} = \delta^c_a~.
\]
%%%%%%
In most cases of interest, the underlying vector space is a pseudoreal representation $R$ of a given compact Lie group $G_c$.

The symplectic boson VOA ${\rm Sb}[V]$ associated to $V$ is freely generated by $\dim V$ Grassmann-even fields $q^a(z)$ whose OPEs are given by
%%%%%%
\[
    q^a(z) q^b(w) = \frac{\Omega^{ab}}{z-w} + :\! q^a(z) q^b(w)\!:~.
\]
%%%%%%
The mode expansion for these generating fields is given by
%%%%%%
\[
    q^a(z) = \sum_{n \in \Z} q^a_{n}~z^{-n-1}~,
\]
%%%%%%
for which the above OPEs translate to the following Lie bracket,
%%%%%%
\[
    [q^a_{n}, q^b_{m}] = \Omega^{ab} \delta_{n+m+1,0}~.
\]
%%%%%%
We equip ${\rm Sb}[V]$ with the following stress tensor/Virasoro field:
%%%%%%
\[
    T(z) = \tfrac{1}{2} \Omega_{ab} :\! q^b \partial q^a\!:(z)~.
\]
%%%%%%
This stress tensor assigns $q^a$ conformal weight $\frac{1}{2}$, and hence ${\rm Sb}[V]$ is a VOA of strong CFT type. In this case, the additional $\Z$ grading coming from four-dimensional ${\rm U}(1)_r$ symmetry is trivial, so the $\Z_2$ grading in the definition of graded unitarity is identified with the half-integrality of conformal weights.

We now equip ${\rm Sb}[V]$ with a suitable good filtration $\mathfrak{R}_\bullet$. As this comes from a free SCFT, it is actually straightforward to define the corresponding $\mathfrak{R}$-grading directly. The physically-motivated grading arises from assigning the generators $q^a$ $R=\frac{1}{2}$, corresponding to the fact that the corresponding physical fields have $R$-charge $\frac{1}{2}$. We thus grade the vector space underlying ${\rm Sb}[V]$ according to
%%%%%%
\begin{equation}\label{eq:Sbgrading}
    {\rm Sb}[V]_{\scriptstyle{\frac{\ell}{2}}} \coloneqq
    \begin{cases}
        {\rm span}\bigg\{q^{a_1}_{-n_1 - 1} \dots q^{a_\ell}_{-n_\ell - 1}|0\rangle \bigg\} & \ell \geqslant 0\\
        0 & \ell < 0
    \end{cases}~,
\end{equation}
%%%%%%
and the induced $\mathfrak{R}$-filtration is given by
%%%%%%
\[
    \mathfrak{R}_{p + \alpha} {\rm Sb}[V] \coloneqq \bigoplus_{\substack{p' \in \Z \\ p' \leqslant p}} {\rm Sb}[V]_{p'+\alpha}~, \qquad p\in \Z~, \qquad \alpha = 0, \tfrac{1}{2}~,
\]
%%%%%%
\emph{cf.} Eq. (3.41) of \cite{Beem:2019tfp}. The action of $s$ is determined by Eq. \eqref{eq:Sbgrading}: it acts as $(-1)^{2R}$ on ${\rm Sb}[V]_R$, which is the same as $(-1)^{2h}$.

Following the physical realization of the conjugation $\rho$ described in \cite{ArabiArdehali:2025fad}, we are led to the following action on the generators,
%%%%%%
\begin{equation}\label{eq:Sbconjugation}
    \rho(q^a) = -\Omega_{ab} q^b~,
\end{equation}
%%%%%%
which generalizes to an action on standard basis elements as follows,
%%%%%%
\[
    \rho\bigg(q^{a_1}_{-n_1 - 1} \dots q^{a_\ell}_{-n_\ell - 1}|0\rangle\bigg) = (-1)^\ell \Omega_{a_1 b_1} \dots \Omega_{a_\ell b_\ell} q^{b_1}_{-n_1 - 1} \dots q^{b_\ell}_{-n_\ell - 1}|0\rangle~,
\]
%%%%%%
extended to all of ${\rm Sb}[V]$ by anti-linearity. It is clear that $\rho$ preserves $\mathfrak{R}$-degree and commutes with $s$. That it satisfies
%%%%%%
\[
    \rho (x_n y) = \rho(x)_n \rho(y)~,
\]
%%%%%%
for any $x, y \in {\rm Sb}[V]$, is equally immediate.

%%%%%%
\begin{prop}\label{prop:Sbgradedunitarity}
    The filtration $\mathfrak{R}_\bullet$ is a nondegenerate half-integral good filtration and $\rho$ defines a quaternionic structure. Moreover, the pairing $\langle x|x\rangle = ((\sigma \circ \rho)(x), x)$ is positive-definite, hence this data equips ${\rm Sb}[V]$ with the structure of a graded-unitary VOA.
\end{prop}
%%%%%%

%%%%%%
\begin{proof}
    The goodness of the filtration is straightforward. Nondegeneracy of $(-,-)$ follows by noting that $(\rho(x), x)$ is nonzero for every $x$, which follows from a direct computation. We note that for any $x \in {\rm Sb}[V]_R$, the pairing $(\rho(x), x)$ is $(-\ii)^{2R}$ times a positive real number.
	
    The grading on ${\rm Sb}[V]$ that comes from refining $\mathfrak{R}_\bullet$ by Gram--Schmidt is precisely the grading given in the first instance in Eq. \eqref{eq:Sbgrading}. In particular, we find that $\sigma$ acts as multiplication by $\ii^{2R}$ on ${\rm Sb}[V]_R$, so $\rho$ indeed defines a quaternionic structure on ${\rm Sb}[V]$. Finally, the action of $\sigma$ precisely removes the factor of $(-\ii)^{2R}$ in the pairing $(\rho(x), x)$ and hence the Hermitian form $\langle - | - \rangle$ is positive-definite.
\end{proof}
%%%%%%

%%%%%%
\begin{rmk}
    We note that the adjoint operation on the mode $q^a_{-n-1}$ takes the form
    %%%%%%
    \[
        (q^a_{-n-1})^\dagger =
        \begin{cases}
            \Omega_{ab} q^b_{n} & n \geqslant 0\\
            -\Omega_{ab} q^b_{n} & n < 0
	\end{cases}~.
    \]
    In particular, viewing $q^a_{-n-1}$ with $n \geqslant 0$ as a creation operator, this simply exchanges creation and annihilation operators.
\end{rmk}
%%%%%%

%%%%%%%%%%%%%%%%%%%%%%%%%%%%%%%%%%%%%%%%%%%%%%%%%%%%%%%%%%%%%%%%
\subsubsection{\label{sec:ex2}Example 2: symplectic fermion VOA}
%%%%%%%%%%%%%%%%%%%%%%%%%%%%%%%%%%%%%%%%%%%%%%%%%%%%%%%%%%%%%%%%

We next consider the symplectic fermion VOA, certain instances of which arises in the SCFT/VOA correspondence from free vector multiplets. Let $V$ again be a symplectic vector space; we will use the same notation as in the previous example. The symplectic fermion VOA ${\rm Sf}[V]$ is freely generated by $\dim V$ Grassmann-odd fields $\eta^a(z)$ with OPEs given by
%%%%%%
\[
    \eta^a(z) \eta^b(w) = \frac{\Omega^{ab}}{(z-w)^2} + :\!\eta^a(z) \eta^b(w)\!:~.
\]
%%%%%%
The mode expansion for these generating fields is expressed as
%%%%%%
\[
    \eta^a(z) = \sum_{n \in \Z} \eta^a_{n}~z^{-n-1}~,
\]
%%%%%%
from which the above OPEs translate to the following Lie algebra of modes%
\footnote{We use the notation that the bracket $[x,y]$ corresponds to the graded commutator; it is a commutator if at least one of $x$ or $y$ is bosonic and an anticommutator otherwise.}%
%
%%%%%%
\[
    [\eta^a_n, \eta^b_m] = n \Omega^{ab} \delta_{n+m,0}~.
\]
%%%%%%
${\rm Sf}[V]$ can be equipped with the stress tensor/Virasoro field
%%%%%%
\[
    T(z) = \tfrac{1}{2}\Omega_{ab} :\! \eta^b \eta^a\!:(z)
\]
%%%%%%
making it a VOA of strong CFT type.

The graded-unitary structure for this case can be treated quite similarly to the one for ${\rm Sb}[V]$ so we will be brief. As for symplectic bosons, the physical $\mathfrak{R}$-filtration is induced by placing the generators $\eta^a$ in $\mathfrak{R}$-degree $\frac{1}{2}$, corresponding to the $R$-charge assignments of the relevant gaugino operators in four dimensions. The underlying vector space in this case can again be graded from the get-go,
%%%%%%
\begin{equation}\label{eq:Sfgrading}
    {\rm Sf}[V]_{\scriptstyle{\frac{\ell}{2}}} \coloneqq
    \begin{cases}
        {\rm span}\bigg\{\eta^{a_1}_{-n_1-1} \dots \eta^{a_\ell}_{-n_\ell-1}|0\rangle \bigg\}~,&\quad \ell \geqslant 0\\
	0~, & \quad \ell < 0
    \end{cases}~.
\end{equation}
%%%%%%
and the induced $\mathfrak{R}$-filtration is given by
%%%%%%
\begin{equation*}
    \mathfrak{R}_{p + \alpha} {\rm Sf}[V] \coloneqq \bigoplus_{\substack{p' \in \Z \\ p' \leqslant p}} {\rm Sf}[V]_{p'+\alpha}~, \qquad p\in \Z~, \quad \alpha \in\{0, \tfrac{1}{2}\}~.
\end{equation*}
%%%%%%
The physical conjugation acts on the generating fields as
%%%%%%
\begin{equation}\label{eq:Sfconjugation}
    \rho(\eta^a) = -\ii~\Omega_{ab}\eta^{a}~,
\end{equation}
%%%%%%
and the action of $\sigma$ and $s$ is determined by Eq. \eqref{eq:Sfgrading}: $\sigma$ (resp. $s$) acts as $\ii^{2R}$ (resp. $(-1)^{2R}$) on ${\rm Sf}[V]_R$. By an analysis that proceeds along essentially the same lines as for symplectic bosons, we find the following.

%%%%%%
\begin{prop}\label{prop:Sfgradedunitarity}
    This data equips ${\rm Sf}[V]$ with the structure of a graded-unitary VOA.
\end{prop}
%%%%%%

%%%%%%
\begin{rmk}
    The adjoint of $\eta^a_{-n-1}$ takes the form
    \[
        (\eta^{a}_{-n-1})^\dagger =
        \begin{cases}
            \Omega_{ab} \eta^{b}_{n+1}& n \geqslant 0\\
            -\Omega_{ab} \eta^{b}_{n+1}& n < 0
	\end{cases}~.
    \]
    Note that $\eta^{a}_0$ acts as zero on the entire vacuum module.
\end{rmk}
%%%%%%

%%%%%%
\begin{rmk}
    For a general choice of $V$ there is no natural, nontrivial internal $\Z$ grading on ${\rm Sf}[V]$ that refines the $\mathbb{Z}_2$ Grassmann parity. In the cases of interest, the vector space $V$ takes the form $\C^2 \otimes \fg$ where $\fg = \fg_c \otimes_\R \C$ is the complexification of the Lie algebra $\fg_c = {\rm Lie}(G_c)$ of a compact Lie group $G_c$; the symplectic form is simply the tensor product of the standard symplectic form on $\C^2$ and a preferred symmetric nondegenerate invariant bilinear form on $\fg$, \emph{e.g.} the Killing form when $\fg$ is semisimple. In this special case, the vector space can be split as $\C^2 \otimes \fg = \fg_+ \oplus \fg_-$ where $\fg_\pm$ are given internal weights $d = \pm 1$, corresponding to the ${\rm U}(1)_r$ charges $\pm \frac{1}{2}$ of the corresponding gaugino operators in four dimensions.

    More generally, one may take $V = \C^2 \otimes U = U_+ \oplus U_-$ where $U$ is the complexification of a real inner product space, with symplectic form determined by the standard symplectic form on $\C^2$ and the metric on $U$. If $\{T_A\}_{A=1}^{\dim U}$ is a basis for $U$ then we denote the generating fields as $\eta^{\alpha, A}$, where $\alpha = \pm$, with OPE given by
    %%%%%%
    \[
        \eta^{\alpha, A}(z) \eta^{\beta, B}(w) = \frac{\epsilon^{\alpha \beta} K^{AB}}{(z-w)^2} + :\!\eta^{\alpha, A}(z) \eta^{\beta, B}(w)\!:~,
    \]
    %%%%%%
    where $\epsilon^{\alpha \beta}$ is the two-index Levi--Civita tensor, defined with $\epsilon^{+-} = 1$, and $K^{AB}$ are the matrix elements of the inverse of the metric in the basis $\{T_A\}$. The generators $\eta^{\pm, A}$ can be given internal degree $\pm 1$, although there may be more flexibility depending on the form of the metric.
\end{rmk}
%%%%%%

%%%%%%
\begin{rmk}\label{rmk:sl2}
    The symplectic fermion VOA $\rm{Sf}[V]$ inherits a natural action of ${\rm Sp}(V)$ by \emph{outer} automorphisms. In the cases of interest to four-dimensional SCFTs, where we take $V = \C^2 \otimes \fg$, it is natural to consider the subgroup ${\rm Sp}(2,\C) \otimes G$, where $G$ acts on $\fg$ via the adjoint action. The action of this reduced structure group will play an important role in our analysis; the action of $G$ will be used to define relative semi-infinite cochain complexes whereas ${\rm Sp}(2,\C)$ will act non-trivially on the differential of this complex. Only the intersection of these automorphisms with the unitary group ${\rm U}(V)$ will be compatible with the Hermitian form $\langle - | - \rangle$.
    
    The generators of the ${\rm Sp}(2, \C)$ action on the vacuum module ${\rm Sf}[\C^2 \otimes U]$ can be made quite explicit in terms of the modes $\eta^{\alpha, A}_n$ of the generating fields. Explicitly, the generator of the diagonal Cartan giving $\eta^{\pm,A}$ weight $\mp 1$ is given by
    %%%%%%
    \[
        \Pi = \sum_{n > 0} \frac{1}{n}K_{AB} \left(\eta^{-,A}_{-n} \eta^{+,B}_n + \eta^{+,A}_{-n} \eta^{-,B}_n\right)~,
    \]
    %%%%%%
    where $K_{AB}$ are the matrix elements of the metric on $U$. The remaining generators are given by
    %%%%%%
    \[
        L = - \sum_{n > 0} \frac{1}{n} K_{AB} \eta^{-,A}_{-n} \eta^{-,B}_n~, \qquad \Lambda = \sum_{n > 0} \frac{1}{n} K_{AB} \eta^{+,A}_{-n} \eta^{+,B}_n~.
    \]
    %%%%%%
    The adjoint of each of these endomorphisms takes the form
    %%%%%%
    \[
        \Pi^\dagger = \Pi~, \qquad L^\dagger = \Lambda~, \qquad \Lambda^\dagger = L~,
    \]
    %%%%%%
    which gives a real form corresponding to the unitary subgroup ${\rm USp}(2) \simeq {\rm SU}(2)$.
\end{rmk}
%%%%%%
    %!TEX Root = ../formality.tex

%----------------------------------------------------------------------%
\section{\label{sec:BRST}Relative semi-infinite cohomology of graded-unitary vertex algebras}
%----------------------------------------------------------------------%

In this section we turn to our main object of study: the relative semi-infinite cohomology of (weak) graded-unitary vertex algebras with (twice critical) affine $\mathfrak{g}$ symmetry.%
\footnote{When $\fg$ has abelian factors, we take twice critical to mean vanishing level for those factors.} %
We first review the notion of relative semi-infinite cohomology and its appearance in four-dimensional $\CN=2$ superconformal gauge theories; see Section 3.4 of \cite{Beem:2013sza} for further physical details. We then turn to structural properties that emerge in the presence of a (weak) graded unitarity structure. We assume all M\"{o}bius vertex algebras appearing in this section are of CFT type unless otherwise stated.

%----------------------------------------------------------------------------------------%
\subsection{\label{sec:BRSTreview}Lightning review of (relative) semi-infinite cohomology}
%----------------------------------------------------------------------------------------%

Choose a compact Lie group $G_c$ with Lie algebra $\fg_c = {\rm Lie}(G_c)$; we denote by $\fg \coloneqq \fg_c \otimes_\R \C$ its complexification, a complex reductive Lie algebra. Let $\{T_A\}_{A = 1}^{\dim \fg}$ be a basis of $\fg_c$ and denote by $f_{AB}^C$ the (real) structure constants of $\fg_c$ with respect to this basis
%%%%%%
\[
	[T_A, T_B] = f^C_{AB} T_C~.
\]
%%%%%%
We denote by $\kappa(-,-)$ a nondegenerate invariant bilinear form on $\fg_c$ that restricts to the Killing form on each of its simple factors (normalized so that long roots have squared-length two) and set $K_{AB} = \kappa(T_A, T_B)$; we denote by $K^{AB}$ its inverse
%%%%%%
\[
	K_{AB} K^{BC} = \delta^C_A~.
\]
%%%%%%
We will use the notation
%%%%%%
\[
	f_{ABC} \coloneqq f^D_{AB} K_{DC} = - f_{ACB}~.
\]
%%%%%%
Denote by $V^k(\mathfrak{g})$ the universal affine vertex algebra associated to $\mathfrak{g}$ at level $k$. This is generated by currents $J_A(z)$ with defining OPE
%%%%%%
\begin{equation}
    J_A(z)J_B(w) = \frac{k K_{AB}}{(z-w)^2}+\frac{f_{AB}^C J_C(w)}{z-w} + :\! J_A(z) J_B(w)\!:~.
\end{equation}
%%%%%%
Following \cite{Arakawa:2018egx}, we make the following definition.
%%%%%%
\begin{dfn}
    Let $\CV$ be a vertex algebra. A $\wh{\fg}_{k}$ \defterm{chiral quantum moment map} when it receives a vertex algebra homomorphism
    %%%%%%
    \[
        \mu_{\CV}:V^k(\mathfrak{g})\to\mathcal{V}~.
    \]
    %%%%%%
    We say that $\CV$ has a \defterm{Hamiltonian} $\wh{\fg}_k$ \defterm{action} if it has a $\wh{\fg}_k$ chiral quantum moment map.
\end{dfn}
%%%%%%

%%%%%%
\begin{rmk}
    When a chiral quantum moment map exists, we also denote by $J_A$ the images of the generating currents of $V^k(\fg)$ in $\CV$.
\end{rmk}
%%%%%%

Let $\CV_{\rm matter}$ be a vertex algebra with a Hamiltonian $\wh{\fg}_{-2h^\vee}$ action, where $h^\vee$ is the dual Coxeter number of $\fg$.%
\footnote{If $\fg$ is not simple we take the levels to satisfy this constraint for each simple factor and to be vanishing on each abelian factor.} %
We then introduce the vertex algebra
%%%%%%
\[
    C(\wh{\fg}_{-2h^\vee}, \CV_{\rm matter}) \coloneqq \CV_{\rm matter} \otimes \bigwedge{}^{\scriptstyle{\frac{\infty}{2}}+\bullet} \fg~,
\]
%%%%%%
where $\bigwedge^{\scriptstyle{\frac{\infty}{2}}+\bullet} \fg$ is the vertex algebra freely generated by $2 \dim \fg$ Grassmann-odd fields $b^A$, $c^A$ with the only singular OPEs given by
%%%%%%
\[
	c^A(z) b^B(w) = \frac{K^{AB}}{z-w} + :\! c^A(z) b^B(w)\!:~,
\]
%%%%%%
%
also called a $bc$ ghost system valued in $\fg$. Note that $C(\wh{\fg}_{-2h^\vee}, \CV_{\rm matter})$ has a natural Hamiltonian $\wh{\fg}_0$ action whose chiral quantum moment map has its image is generated by
%%%%%%
\[
	J^{\rm tot}_A(z) = \overset{J^{\rm gh}_A(z)}{\overbrace{f_{ABC} :\!c^B b^C\!:(z)}} + J_{A}(z) = \sum_{n \in \Z} J^{\rm tot}_{A,n} z^{-n-1}~,
\]
%%%%%%
where $J_{A}$ generate the image of the $\wh{\fg}_{-2h^\vee}$ chiral quantum moment map for $\CV_{\rm matter}$. We can equip $C(\wh{\fg}_{-2h^\vee}, \CV_{\rm matter})$ with a differential given by the zero mode 
%%%%%%
\[
	\bQ = \mathcal{Q}_{0}~, \qquad \mathcal{Q}(z) = :\! c^A\big(J_{A} + \tfrac{1}{2} J_{{\rm gh},A}\big)\!:(z)~.
\]
%%%%%%

%%%%%%
\begin{rmk}
    When $\CV_{\rm matter}$ is a VOA, we equip $C(\wh{\fg}_{-2h^\vee}, \CV_{\rm matter})$ with the following stress tensor/Virasoro field:
    %%%%%%
    \[
    	T(z) = K_{AB}:\!b^A \partial c^B\!:(z) + T_{\rm matter}(z)~,
    \]
    %%%%%%
    where $T_{\rm matter}(z)$ is the stress tensor of $\CV_{\rm matter}$. This stress tensor gives $b^A$ conformal weight one and $c^B$ conformal weight zero. More generally, if $\CV_{\rm matter}$ is a M\"{o}bius vertex algebra, then $C(\wh{\fg}_{-2h^\vee}, \CV_{\rm matter})$ can also be given the structure of a M\"{o}bius vertex algebra. In either case, $C(\wh{\fg}_{-2h^\vee}, \CV_{\rm matter})$ is thus \emph{not} of CFT type.
\end{rmk}
%%%%%%

We now in introduce the vertex algebra of interest: it is the subcomplex of $C(\wh{\fg}_{-2h^\vee}, \CV_{\rm matter})$ relative to the action of the finite (conformal weight zero) part,
%%%%%%
\begin{equation}\label{eq:Vdef}
    C(\wh{\fg}_{-2h^\vee}, \fg, \CV_{\rm matter}) \coloneqq \left\{x \in C(\wh{\fg}_{-2h^\vee}, \CV_{\rm matter}) \big| b^A_0 x = 0 \text{ and } J^{\rm tot}_{A, 0} x = 0 \text{ for all } A\right\}~.
\end{equation}
%%%%%%

%%%%%%
\begin{dfn}
    The \defterm{relative semi-infinite} $\wh{\fg}_{-2h^\vee}$\defterm{-cohomology} with coefficients in $\CV_{\rm matter}$, denoted $H^{\scriptstyle{\frac{\infty}{2}}+\bullet}(\wh{\fg}_{-2h^\vee}, \fg, \CV_{\rm matter})$, is the cohomology of $(C(\wh{\fg}_{-2h^\vee}, \fg, \CV_{\rm matter}), \bQ)$.
\end{dfn}
%%%%%%

As this vertex algebra is the main object we will be studying, several remarks are in order.

%%%%%%
\begin{physrmk}
    As described in \cite{Beem:2013sza}, the relative subcomplex $C(\wh{\fg}_{-2h^\vee}, \fg, \CV_{\rm matter})$ describes the algebra of Schur operators in the limit of vanishing gauge coupling. This algebra of gauge-invariant operators gets deformed by the introduction of the differential $\bQ$ as the gauge coupling is turned on. In other words, this gauging is realized at the level of vertex algebras as the relative semi-infinite $\wh{\fg}_{-2h^\vee}$-cohomology of $\CV_{\rm matter}$. Physically, the vertex algebra $\CV_{\rm matter}$ is thought of as the vertex algebra associated to the to-be-gauged $\CN=2$ SCFT; the requirement on its level translates to the fact that the gauged theory continues to be $\CN=2$ superconformal.
\end{physrmk}
%%%%%%

%%%%%%
\begin{rmk}
    The first condition $b^A_0 x = 0$ in the definition of the relative subcomplex can be viewed as treating $\partial c^A$ as a generator. With this in mind, the Grassmann-odd fields $b^A$ and $\partial c^A$ are on equal footing in the relative subcomplex: they both have conformal weight one and transform in the adjoint representation of $\fg$. It is then convenient to adopt notation that reflects this symmetry between them:
    %%%%%%
    \[
        \eta^{+,A} \coloneqq b^A~, \qquad \eta^{-,A} \coloneqq \partial c^A~,
    \]
    %%%%%%
    where these fields realize the symplectic fermion vertex algebra ${\rm Sf}[\C^2 \otimes \fg]$ described in Section \ref{sec:ex2}. In other words, we can identify the relative subcomplex as
    %%%%%%
    \[
        C(\wh{\fg}_{-2h^\vee}, \fg, \CV_{\rm matter}) \cong \left( \CV_{\rm matter} \otimes {\rm Sf}[\C^2 \otimes \fg] \right)^G~.
    \]
    %%%%%%
    When $\CV_{\rm matter}$ is a VOA, the (state corresponding to the) stress tensor $T(z)$ belongs to this relative subcomplex and so it is naturally a VOA as well. If $\CV_{\rm matter}$ is of (strong) CFT type, so too is $C(\wh{\fg}_{-2h^\vee}, \fg, \CV_{\rm matter})$. Moreover, this stress tensor $T(z)$ (when present) is $\bQ$-closed and not exact,so survives passage to cohomology of the complex.
\end{rmk}
%%%%%%

%%%%%%
\begin{prop}\label{prop:QpQm}
    The vertex algebra $C(\wh{\fg}_{-2h^\vee}, \fg, \CV_{\rm matter})$ has two commuting differentials $\bQ^\pm$, where $\bQ = \bQ^-$. Moreover, the cohomologies defined with respect to either of $\bQ^+$ and $\bQ^-$ are isomorphic.
\end{prop}
%%%%%%

%%%%%%
\begin{proof}
    As mentioned in Section \ref{sec:ex2}, the symplectic fermion VOA ${\rm Sf}[\C^2 \otimes \fg]$ has a ${\rm Sp}(2, \C) \times G$ outer automorphism. Whereas the second factor is used in defining the relative subcomplex, the first factor does not preserve the differential $\bQ$. Instead, this differential fits into the standard representation of ${\rm Sp}(2,\C)$. In terms of the modes of the symplectic fermions and the currents $J_A$ these differentials take the following form:
    %%%%%%
    \begin{equation}\label{eq:Qpm}
    	\bQ^\pm = \sum_{n \in \Z_{\neq 0}} \frac{1}{n} :\! \eta^{\pm,A}_{-n} J_{A,n} \!: + \sum_{\substack{m,n \in \Z_{\neq 0}\\ m \neq n}} \frac{f_{ABC}}{2mn} :\!\eta^{\pm,A}_{-n} \eta^{\pm,B}_{m} \eta^{\mp,C}_{n-m} \!:~.
    \end{equation}
    %%%%%%
    As each of the differentials $\bQ^\pm$ is square-zero and they transform in the standard representation of ${\rm Sp}(2, \C)$ we conclude that they necessarily (anti)commute with one another. As these differentials are exchanged by an automorphism of the underlying vertex algebra, their cohomologies are isomorphic to one another.
\end{proof}
%%%%%%

%%%%%%
\begin{physrmk}
    See \emph{e.g.} \cite[Section 3.4.1]{Beem:2013sza} for a physical discussion of this observation. Roughly speaking, the existence of these two differentials is due to the fact that the vertex algebra can be defined by two different nilpotent elements in the $\CN=2$ superconformal algebra. Although these elements do not commute in the full superconformal algebra, their anti-commutator is a sum of quantized charges which cannot change as the gauge coupling is tuned away from zero. Namely, the restrictions of these supercharges to Schur operators at zero coupling, \emph{i.e.} the operators in the vector space $\CV$, necessarily commute, and this continues to be the case away from zero coupling.
\end{physrmk}
%%%%%%

%%%%%%
\begin{rmk}
    If we grade $C(\wh{\fg}_{-2h^\vee}, \fg, \CV_{\rm matter})$ by $\eta^{-,A}$-degree minus $\eta^{+,A}$ degree, $\bQ^-$ will be homogeneous of degree $1$ and $\bQ^+$ will be homogeneous of degree $-1$; each will equip $\CV$ with the structure of a differential graded (DG) vertex algebra. This grading is inherited from the endomorphism $\Pi$ described in Remark \ref{rmk:sl2}, \emph{i.e.} the diagonal Cartan subalgebra of the ${\rm Sp}(2,\C)$ automorphisms of the symplectic fermions ${\rm Sf}(\C^2 \otimes \fg)$.
\end{rmk}
%%%%%%

%%%%%%
\begin{rmk}
    In general there are as many ${\rm Sp}(2,\C)$ automorphisms of the relative subcomplex as the number of simple factors plus the rank of the radical of $\fg$. Correspondingly, we can construct twice as many differentials; all of these differentials necessarily commute with one another. We will return to this observation in Section \ref{sec:automorphisms}.
\end{rmk}
%%%%%%

%--------------------------------------------------------------------------------------------%
\subsection{\label{sec:Kahler}Decompositions from graded unitarity and the K\"{a}hler package}
%--------------------------------------------------------------------------------------------%

We now assume that $\CV_{\rm matter}$ is itself a weak graded-unitary vertex algebra. We view the vertex algebra underlying the relative subcomplex $C(\wh{\fg}_{-2h^\vee}, \fg, \CV_{\rm matter})$ as the $G$-invariant subalgebra of  $\CV_{\rm matter} \otimes {\rm Sf}[\C^2 \otimes \fg]$. The vertex algebra $\CV_{\rm matter} \otimes {\rm Sf}[\C^2 \otimes \fg]$ inherits a weak graded-unitary structure from that of $\CV_{\rm matter}$ and of ${\rm Sf}[\C^2 \otimes \fg]$ as given in Section \ref{sec:ex2}, where the internal $\Z$ grading is taken to be the diagonal; this is naturally $G$-equivariant and hence induces a weak graded-unitary structure on $\CV$. Of course, if $\CV_{\rm matter}$ is a graded-unitary vertex algebra then so too is $C(\wh{\fg}_{-2h^\vee}, \fg, \CV_{\rm matter})$.

%%%%%%
\begin{rmk}
    The internal grading on ${\rm Sf}[\C^2 \otimes \fg]$ induces an additional $\Z$ grading on $\CV$, although it could be redundant if $\CV_{\rm matter}$ is concentrated in degree zero as in the graded-unitary structure on ${\rm Sb}[V]$ described in Section \ref{sec:ex1}; we will return to this additional grading in Section \ref{sec:automorphisms}.
\end{rmk}
%%%%%%

Our analysis will require that the $\wh{\fg}_{-2h^\vee}$ chiral quantum moment map satisfies the following properties.

%%%%%%
\begin{dfn}\label{dfn:goodaction}
    Let $(\CV, (-,-), \mathfrak{R}_\bullet, \rho)$ be a weak graded-unitary vertex algebra equipped with Hamiltonian $\wh{\fg}_k$ action with chiral quantum moment map $\mu: V^k(\fg) \to \CV$ whose image is generated by states $J_{A,-1}|0\rangle$, $A = 1, \dots, \dim \fg$. We say that this Hamiltonian $\wh{\fg}_k$ action is \defterm{good} if
    %%%%%%
    \begin{itemize}
        \item[(i)] $J_{A,-1}|0\rangle \in \mathfrak{R}_1 (\CV_{\rm matter})_{h = 1}$~,
        \item[(ii)] $J_{A, n}^{[1]} = 0$ for all $A = 1, \dots, \dim \fg$ and $n \geq 0$~,
        \item[(iii)] $\rho(J_A) = K^{AB} J_B$~.
    \end{itemize}
    %%%%%%
\end{dfn}
%%%%%%

%%%%%%
\begin{rmk}
    Condition $(i)$ implies that $J_{A, -1}|0\rangle$ belongs to the $s$-invariant subspace of $\CV_{\rm matter}$ and hence $\rho$ acts on these vectors as an anti-linear involution of $\fg$, hence is equivalent to a choice of real form $\fg_c$ thereof. The action of $\rho$ in condition $(iii)$ is chosen to be the Cartan involution and thereby select the compact real form. When the filtration is good, \emph{i.e.}, $\CV_{\rm matter}$ is graded unitary, then condition $(i)$ implies condition $(ii)$, moreover, $J_{A,-1}|0\rangle$ belongs to the graded component with $R = 1$. We note that if $\CV$ is a unitary M\"{o}bius vertex algebra with a unitary Hamiltonian $\wh{\fg}_k$ action then this action is necessarily good with respect to the weak graded-unitary structure described in Section \ref{sec:ex0}.
\end{rmk}
%%%%%%

%%%%%%
\begin{physrmk}
    In the context of the SCFT/VOA correspondence, the physical operators giving rise to these currents always have conformal weight $h = 1$, degree $d = 0$, and $R$-charge $R = 1$. The above action of $\rho$ is taken to match one coming from this four-dimensional origin.%
    \footnote{Our formula for the action of $\rho$ differs from the one appearing in \cite{ArabiArdehali:2025fad} as we work in the convention that there is no $\ii$ appearing in the definition of the real structure constants, \emph{i.e.}, we work with anti-Hermitian generators rather than the Hermitian generators. As $\rho$ is anti-linear, the additional factor of $\ii$ appearing in the basis accounts for the lack of a minus sign in the action on $J_A$.} %
\end{physrmk}
%%%%%%

%%%%%%
\begin{physrmk}
    The goodness condition on the Hamiltonian $\wh{\fg}_k$ action is satisfied by many vertex algebras of physical interest outside of four dimensions. For example, the VOAs described in \cite{Costello:2018fnz} in the context of $A$-twisted (also called $H$-twisted) three-dimensional $\CN=4$ gauge theories are given by relative semi-infinite cohomology of free-field VOAs built from symplectic bosons and free fermions. Moreover, there is a nondegenerate half-integral filtration on these VOAs coming from the ${\rm SU}(2)_H$ $R$-symmetry but the free fermions have vanishing $R$-charge, \emph{i.e.}, filtration degree, and hence this filtration is not good. Nonetheless, this Hamiltonian $\wh{\fg}_{-2h^\vee}$ action is good because the only contributions to $J_{A,n}^{[1]}$ can come from the symplectic bosons and hence vanish due to the goodness of the ${\rm SU}(2)_H$-filtration upon restricting to the symplectic bosons.
\end{physrmk}
%%%%%%

With the above hypotheses on the Hamiltonian $\wh{\fg}_{-2h^\vee}$ action, we have the following lemma.

%%%%%%
\begin{lem}
	Let $\CV_{\rm matter}$ be a weak graded-unitary vertex algebra with a good Hamiltonian $\wh{\fg}_{-2h^\vee}$ action. Then the differentials $\bQ^\pm$ on $C(\wh{\fg}_{-2h^\vee}, \fg, \CV_{\rm matter})$ can be written as
    %%%%%%%
    \[
        \bQ^\pm = Q^\pm + S^\pm
    \]
    %%%%%%%
    where $Q^\pm$ and $S^\pm$ are homogeneous of $\mathfrak{R}$-degree $R = \frac{1}{2}$ and $R = -\frac{1}{2}$, respectively. The $Q^\pm$, $S^\pm$ are all square-zero and satisfy
	%%%%%%
    \[
		[Q^\pm, S^\pm] = 0~, \qquad [Q^+, Q^-] = 0 = [S^+, S^-]~, \qquad [Q^+, S^-] = - [Q^-, S^+]~.
	\]
    %%%%%%
    Moreover, they satisfy
    %%%%%%
    \[
    	(S^\alpha)^\dagger = -\epsilon_{\alpha \beta} Q^\beta \qquad (Q^\alpha)^\dagger = \epsilon_{\alpha \beta} S^\beta~.
    \]
    %%%%%%
\end{lem}
%%%%%%

%%%%%%
\begin{rmk}
    If $\CV_{\rm matter}$ is a unitary M\"{o}bius algebra with a unitary Hamiltonian $\wh{\fg}_{-2h^\vee}$ action, then the above decomposition of $\bQ^-$---applied to the weak graded-unitary structure described in Section \ref{sec:ex0}---is precisely the one appearing in \cite{FrenkelGarlandZuckerman1986}.
\end{rmk}
%%%%%%

%%%%%%
\begin{physrmk}
    We identify $Q^-$ and $Q^+$ with the Poincar\'{e} supercharges $Q^1_-$ and $-\wt{Q}_{2\dot-}$ and the $S^+$, $S^-$ with the superconformal charges $\wt{S}^{2\dot-}$ and $S^-_1$. More precisely, we identify them with the restriction these supercharges and superconformal charges to the vector space of Schur operators at zero coupling. Equivalently, we can view them as the corrections to these charges for non-zero gauge coupling, so that, \emph{e.g.}, $Q^- = Q^{1(1)}_-$ and $Q^+ = -\wt{Q}^{(1)}_{2,\dot-}$. The way these operators are exchanged by ${}^\dagger$ realizes the conjugation properties of the Poincar\'{e} supercharges and their paired superconformal charges in radial quantization.
\end{physrmk}
%%%%%%

%%%%%%
\begin{proof}
    The shortening condition (Proposition \ref{prop:shortening}) of nondegenerate half-integral filtrations together with the goodness hypothesis on the action implies that the modes $J_{A,n}$ decompose as
    %%%%%%
    \[
    	J_{A,n} = J_{A,n}^{[1]} + J_{A,n}^{[0]} + J_{A,n}^{[-1]}~,
    \]
    %%%%%%
    where $J_{A,n}^{[1]} = 0$ if $n \geqslant 0$. Moreover, as $J_{A,n}^{[R]}$ and $J_{A, -n}^{[-R]}$ are adjoint to one another, this further implies that $J_{A,n}^{[-1]} = 0$ if $n \leqslant 0$.
    
    We now turn to the differentials $\bQ^\pm$, which are built from the modes $J_{A,n}$ and $\eta^{\alpha, A}_n$ of the currents and symplectic fermions. When they are combined into the differential $\bQ^\pm$, their modes are paired in such a way that the whole differential can either increase $\mathfrak{R}$-degree by $\frac{1}{2}$ or decrease it by $\frac{1}{2}$. We can thus write
    %%%%%%
    \[
    	\bQ^\pm = Q^\pm + S^\pm
    \]
    %%%%%%
    where $Q^\pm$ increases $\mathfrak{R}$-degree and $S^\pm$ decreases $\mathfrak{R}$-degree. Explicitly, $Q^\pm$ is given by
    %%%%%%
    \begin{equation}\label{eq:BRSTQ}
        \begin{aligned}
            Q^\pm & = \sum_{n > 0} \tfrac{1}{n}\bigg(\eta^{\pm, A}_{-n} J^{[0]}_{A,n} - J^{[1]}_{A,-n} \eta^{\pm, A}_{n}\bigg)\\
            & \qquad + \sum_{m,n > 0}f_{ABC} \left(\tfrac{1}{n(m+n)}\eta^{\pm, A}_{-n} \eta^{\mp,B}_{-m}\eta^{\pm,C}_{n+m} - \tfrac{1}{2mn}\eta^{\pm, A}_{-n} \eta^{\pm,B}_{-m}\eta^{\mp,C}_{n+m}\right)\\
        \end{aligned}
    \end{equation}
    %%%%%%
    and $S^\pm$ is given by
    %%%%%%
    \begin{equation}\label{eq:BRSTS}
        \begin{aligned}
            S^\pm & = \sum_{n > 0} \tfrac{1}{n}\bigg(\eta^{\pm, A}_{-n} J^{[-1]}_{A,n} - J^{[0]}_{A,-n} \eta^{\pm, A}_{n}\bigg)\\
            & \qquad + \sum_{m,n > 0}f_{ABC} \left(\tfrac{1}{m(m+n)}\eta^{\pm, A}_{-n-m} \eta^{\mp,B}_{n}\eta^{\pm,C}_{m} -\tfrac{1}{2mn}\eta^{\mp,A}_{-n-m}\eta^{\pm, B}_{n} \eta^{\pm,C}_{m}\right)~.\\
        \end{aligned}
    \end{equation}
    %%%%%%
    As $\bQ^\pm$ is square-zero, these two components are necessarily square-zero and $Q^\pm$ (anti\=/)commutes with $S^\pm$. The remaining (anti\=/)commutation relations follow from the fact that $\bQ^\pm$ (anti\=/)commute with one another.

    We now turn to the adjoint properties of these operators. With the prescribed action of $\rho$ on $J_A$, the adjoint of $J_{A,n}$ is given by
    %%%%%%
    \[
        J_{A,n}{}^\dagger = -K^{AB} \sigma^{-1} J_{B,-n} \sigma~.
    \]
    %%%%%%
    In terms of the homogeneous components, this corresponds to
    %%%%%%
    \[
        J_{A,n}^{[\pm 1]}{}^\dagger = K^{AB} J_{B,-n}^{[\mp1]}~, \qquad J_{A,n}^{[0]}{}^\dagger = -K^{AB} J_{B,-n}^{[0]}~.
    \]
    %%%%%%
    Using this adjoint of $J_{A,n}^{[0]}$ and $J_{A,n}^{[\pm 1]}$, the claimed adjoints follow from a simple computation.
\end{proof}
%%%%%%

%%%%%%
\begin{cor}
    Denote the adjoint of $\bQ^\pm$ by $\overline{\bQ}_\pm \coloneqq (\bQ^\pm)^\dagger$. Then these four operators satisfy
    %%%%%%
    \[
		[\bQ^\alpha, \bQ^\beta] = 0~, \qquad [\overline{\bQ}_\alpha, \overline{\bQ}_\beta] = 0~,
    \]
    %%%%%%
    and
    %%%%%%
    \[
        [\bQ^\pm, \overline{\bQ}_\mp] = 0~, \qquad [\bQ^+, \overline{\bQ}_+] = [\bQ^-, \overline{\bQ}_-]~.
    \]
    %%%%%%
\end{cor}
%%%%%%

Let us put these observations together. If $\CV_{\rm matter}$ is a weak graded-unitary vertex algebra (of CFT type) with a good Hamiltonian $\wh{\fg}_{-2h^\vee}$ action then $C(\wh{\fg}_{-2h^\vee}, \fg, \CV_{\rm matter})$ is naturally a weak graded-unitary vertex algebra. Moreover, it has a pair of differentials $\bQ^\alpha$, $\alpha = \pm$, their adjoints $\overline{\bQ}_\alpha = (\bQ^\alpha)^\dagger$, and their ``Laplacian'' $\Delta \coloneqq [\bQ^\pm, \overline{\bQ}_\pm]$ satisfying
%%%%%%
\begin{equation}\label{eq:Hodge}
	[\bQ^\alpha, \bQ^\beta] = 0~, \qquad [\overline{\bQ}_\alpha, \overline{\bQ}_\beta] = 0~, \qquad [\bQ^\alpha, \overline{\bQ}_\beta] = \delta^\alpha_\beta \Delta~.
\end{equation}
%%%%%%
Additionally, the unitary action of ${\rm USp}(2)$ on ${\rm Sf}[\C^2 \otimes \fg]$ by vertex algebra automorphisms described in Remark \ref{rmk:sl2} is inherited by $C(\wh{\fg}_{-2h^\vee}, \fg, \CV_{\rm matter})$. Its action on the differentials $\bQ^\alpha$ is given by
%%%%%%
\begin{equation}\label{eq:Lefschetz1}
    [\Pi, \bQ^\pm] = \mp \bQ^\pm~, \qquad [L, \bQ^+] = \bQ^-~, \qquad [\Lambda, \bQ^-] = \bQ^+~,
\end{equation}
%%%%%%
and on their adjoints by
%%%%%%
\begin{equation}\label{eq:Lefschetz2}
    [\Pi, \overline{\bQ}_\pm] = \pm \overline{\bQ}_\pm~, \qquad [L, \overline{\bQ}_-] = -\overline{\bQ}_+~, \qquad [\Lambda, \overline{\bQ}_+] = -\overline{\bQ}_-~,
\end{equation}
%%%%%%
where the remaining commutators vanish, and hence the Laplacian $\Delta$ is invariant. Standard results of Hodge theory (restricted to subspaces of a fixed conformal weight, which are finite dimensional and are preserved by both $\bQ^\pm$ and $\overline{\bQ}_\pm$) imply that there are two natural Hodge decompositions of $C(\wh{\fg}_{-2h^\vee}, \fg, \CV_{\rm matter})$:
%%%%%%
\[
	C(\wh{\fg}_{-2h^\vee}, \fg, \CV_{\rm matter}) = \ker \Delta \oplus {\rm im} \bQ^+ \oplus {\rm im} \overline{\bQ}_+ = \ker \Delta \oplus {\rm im} \bQ^- \oplus {\rm im} \overline{\bQ}_-~.
\]
%%%%%%
This decomposition, together with the operators $\bQ^\alpha$, $\overline{\bQ}_\alpha$, and $\Delta$, imposes strong constraints on the structure of the $\bQ^-$ (or $\bQ^+$) cohomology in much the same way as de Rham cohomology of compact K\"{a}hler manifolds is constrained. We call this collection of structures the \emph{K\"{a}hler package} for $C(\wh{\fg}_{-2h^\vee}, \fg, \CV_{\rm matter})$, and have the following.%
\footnote{We acknowledge that in more general Hodge-theoretic situations the notion of a K\"{a}hler packages includes additional properties, such as the Lefschetz hyperplane theorem, not all of which have obvious analogues here. We will nevertheless adopt this terminology; it would be interesting to understand to precisely what extent results for K\"ahler manifolds can be extended to the DG vertex algebras underlying relative semi-infinite cohomology of graded-unitary vertex algebras.} %
%

%%%%%%
\begin{thm}\label{thm:BRSTKahlerpackage}
    Let $\CV_{\rm matter}$ be a weak graded-unitary vertex algebra with a good Hamiltonian $\wh{\fg}_{-2h^\vee}$ action. Then $C(\wh{\fg}_{-2h^\vee}, \fg, \CV_{\rm matter})$ is also a weak graded-unitary vertex algebra and, moreover, is equipped with a K\"{a}hler package. If $\CV_{\rm matter}$ is graded unitary, then so too is $C(\wh{\fg}_{-2h^\vee}, \fg, \CV_{\rm matter})$.
\end{thm}
%%%%%%

%%%%%%
\begin{physrmk}
	In light of Proposition \ref{prop:Sbgradedunitarity}, this theorem holds for all examples coming from all Lagrangian $\CN=2$ superconformal gauge theories.
\end{physrmk}
%%%%%%

%%%%%%
\begin{rmk}
    The relations in Eqs. \eqref{eq:Lefschetz1} and \eqref{eq:Lefschetz2} can be viewed as the K\"{a}hler identites expressing the action of the ``Lefschetz'' $\fsl(2)$ triple $\{\Pi, L, \Lambda\}$ on the differentials $\bQ^\pm$, their adjoints $\overline{\bQ}_\pm$, and their Laplacian $\Delta$; indeed, this is the reason for the notation. For example, assuming $\CV_{\rm matter}$ is concentrated in cohomological degree $0$, the weight spaces of $\Pi$ are precisely the subspaces of a fixed cohomological degree in both cases. That said, unlike in K\"{a}hler geometry where the Lefschetz operator is realized by wedge product with the K\"{a}hler form, our ``Lefschetz operator'' $L$ is not realized as a vertex-algebraic product with an element in $C(\wh{\fg}_{-2h^\vee}, \fg, \CV_{\rm matter})$ due to its status as an outer automorphism, as emphasized in Remark \ref{rmk:sl2}.
\end{rmk}
%%%%%%

%-----------------------------------------------------------------------------------------------------%
\subsection{\label{sec:recombination}Recombination, quartet mechanism, and the $\bQ^- \bQ^+$ lemma}
%-----------------------------------------------------------------------------------------------------%

We saw in the previous subsection that when $\CV_{\rm matter}$ is a weak graded-unitary vertex algebra with a good Hamiltonian $\wh{\fg}_{-2h^\vee}$ action, the vertex algebra $ C(\wh{\fg}_{-2h^\vee}, \fg, \CV_{\rm matter})$ is a complex Hilbert space equipped with a unitary action of the ${}^*$-Lie superalgebra with basis $\bQ^\pm$, $\overline{\bQ}_\pm = (\bQ^\pm)^\dagger$, and $\Delta$, with (anti-)commutation relations given in Eq. \eqref{eq:Hodge}. In this subsection we describe the main structural consequence of this observation.

%%%%%%
\begin{physrmk}
    As mentioned in Section \ref{sec:BRST}, the four-dimensional interpretation of the vertex algebra $C(\wh{\fg}_{-2h^\vee}, \fg, \CV_{\rm matter})$ is as Schur operators at zero gauge coupling. The differential, say, $\bQ^-$ describes how this algebra is modified as the gauge coupling is tuned away from zero: the cohomology classes of $\bQ^-$ describe operators that are still of Schur type after turning on the gauge coupling. The elements of $C(\wh{\fg}_{-2h^\vee}, \fg, \CV_{\rm matter})$ that do not survive must be lifted in a very specific way: in quartets as dictated by the recombination of multiplets of Schur type, \emph{cf.} Section 3.4.1 of \cite{Beem:2013sza}.
\end{physrmk}
%%%%%%

%%%%%%
\begin{prop}\label{prop:quartet}
    Let $V$ be a finite-dimensional unitary representation of the ${}^*$-Lie superalgebra in Eq. \eqref{eq:Hodge}. Then $V$ admits an orthogonal decomposition as the direct sum of trivial representations and copies of the four-dimensional representation appearing in Figure \ref{fig:quartet}.
\end{prop}
%%%%%%

%%%%%%
\begin{rmk}
    The nonzero elements of $V$ that belong to the trivial representations will be the $\Delta$-harmonic vectors; these will give rise to nonzero classes in $\bQ^-$- or $\bQ^+$-cohomology. The remaining $\bQ^-$- or $\bQ^+$-closed vectors will residing in four-dimensional representations and will necessarily be exact, hence give rise to trivial cohomology classes.
\end{rmk}
%%%%%%

%%%%%%
\begin{proof}
    We first note that $\Delta$ is a positive semi-definite, self-adjoint operator on a finite-dimensional Hilbert space and hence we have an orthogonal decomposition of $V$ into harmonic eigenspaces:
    %%%%%%
    \[
        V = \bigoplus_{\delta \geqslant 0} V_\delta~.
    \]
    %%%%%%
    Moreover, if $\Delta x = \delta x$ we have
    %%%%%%
    \[
        \delta |\!| x|\!|^2 = \langle x | \Delta x\rangle = \langle \bQ^\pm x | \bQ^\pm x\rangle + \langle \overline{\bQ}_\pm x | \overline{\bQ}_\pm x\rangle = |\!| \bQ^\pm x|\!|^2 + |\!| \overline{\bQ}_\pm x|\!|^2 \geqslant 0~. 
    \]
    %%%%%%
    In particular, $\delta = 0$ if and only if $\bQ^\pm x = 0 = \overline{\bQ}_\pm x$, \emph{i.e.}, $V_0$ admits an orthogonal decomposition into a direct sum of trivial representations.

    If $\delta > 0$ then we can write
    %%%%%%
    \[
        x = \delta^{-1} \left(\bQ^+ \overline{\bQ}_+ x + \overline{\bQ}_+ \bQ^+ x \right) = \delta^{-1} \left(\bQ^- \overline{\bQ}_- x + \overline{\bQ}_- \bQ^- x \right)
    \]
    %%%%%%
    which gives rise to two orthogonal decompositions of $V_\delta$:
    %%%%%%
    \[
        V_\delta = \ker \bQ^+|_{V_\delta} \oplus \ker \overline{\bQ}_+|_{V_\delta} = \ker \bQ^-|_{V_\delta} \oplus \ker \overline{\bQ}_-|_{V_\delta} ~.
    \]
    %%%%%%
    We can combine these into an orthogonal decomposition of the form
    %%%%%%
    \[
    \begin{aligned}
        V_\delta & = \left(\ker \bQ^+|_{V_\delta} \cap \ker \bQ^-|_{V_\delta}\right) \oplus \left(\ker \overline{\bQ}_+|_{V_\delta} \cap \ker \bQ^-|_{V_\delta}\right)\\
        & \qquad \oplus \left(\ker \bQ^+|_{V_\delta} \cap \ker \overline{\bQ}_-|_{V_\delta}\right) \oplus \left(\ker \overline{\bQ}_+|_{V_\delta} \cap \ker \overline{\bQ}_-|_{V_\delta}\right)~.
    \end{aligned}
    \]
    %%%%%%
    For any nonzero $\tilde{x} \in \ker \overline{\bQ}_+|_{V_\delta} \cap \ker \overline{\bQ}_-|_{V_\delta}$, the vector $\bQ^\pm \tilde{x}$ is automatically an element of $\ker \bQ^\pm \cap \ker \overline{\bQ}_\mp$ and is necessarily nonzero because
    %%%%%%
    \[
        |\!| \bQ^\pm \tilde{x} |\!|^2 = \langle \bQ^\pm \tilde{x} | \bQ^\pm \tilde{x} \rangle = \langle \tilde{x} | \overline{\bQ}_\pm \bQ^\pm x \rangle = \langle \tilde{x} | \Delta \tilde{x} \rangle = \delta |\!|\tilde{x} |\!|^2 > 0~.
    \]
    %%%%%%
    Similarly, the vector $\bQ^- \bQ^+x$ is a nonzero element of $\ker \bQ^+ \cap \ker \bQ^-$. The four vectors $\tilde{x}$, $\bQ^\pm x$, and $\bQ^- \bQ^+ x$ form an irreducible four-dimensional representation illustrated in Figure \ref{fig:quartet}. Finally, if $\tilde{x}$ and $\tilde{y}$ are orthogonal elements of $\ker \overline{\bQ}_+|_{V_\delta} \cap \ker \overline{\bQ}_-|_{V_\delta}$ then the resulting four-dimensional representations will be orthogonal.
\end{proof}
%%%%%%

%%%%%%
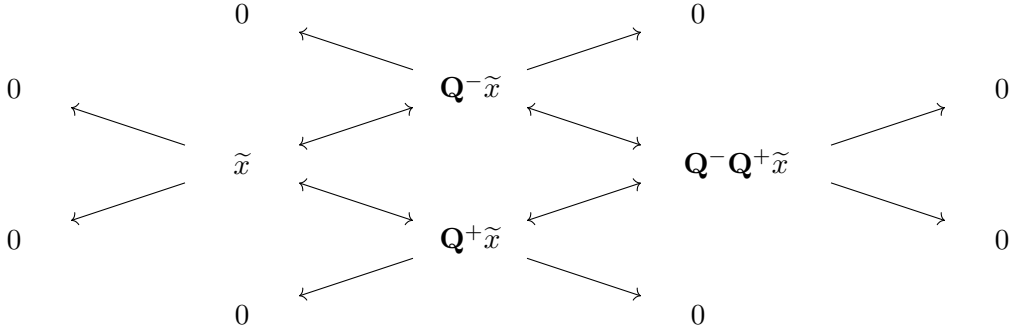
\begin{figure}[ht]
	\begin{tikzpicture}
		\draw (0,0) node {$\wt{x}$};
		\draw (3,1) node {$\bQ^- \wt{x}$};
		\draw (3,-1) node {$\bQ^+ \wt{x}$};
		\draw (6.5,0) node {$\bQ^- \bQ^+ \wt{x}$};
		
		\draw (0,2) node {$0$};
		\draw (0,-2) node {$0$};
		\draw (-3,1) node {$0$};
		\draw (-3,-1) node {$0$};
		\draw (6,2) node {$0$};
		\draw (6,-2) node {$0$};
		\draw (10,1) node {$0$};
		\draw (10,-1) node {$0$};
		
		\draw[->] (-0.75, 0.25) -- (-2.25, 0.75);
		\draw[->] (-0.75, -0.25) -- (-2.25, -0.75);
		\draw[->] (7.75, 0.25) -- (9.25, 0.75);
		\draw[->] (7.75, -0.25) -- (9.25, -0.75);
		
		\draw[<->] (0.75, 0.25) -- (2.25, 0.75);
		\draw[<->] (0.75, -0.25) -- (2.25, -0.75);
		\draw[<->] (5.25, 0.25) -- (3.75, 0.75);
		\draw[<->] (5.25, -0.25) -- (3.75, -0.75);
		
		\draw[<-] (0.75, 1.75) -- (2.25, 1.25);
		\draw[<-] (0.75, -1.75) -- (2.25, -1.25);
		\draw[<-] (5.25, 1.75) -- (3.75, 1.25);
		\draw[<-] (5.25, -1.75) -- (3.75, -1.25);
	\end{tikzpicture}
	\caption{The quartet structure of operators lifted from $\bQ^-$- and $\bQ^+$-cohomology. The rightwards arrows correspond to acting with $\bQ^-$ (upwards) and $\bQ^+$ (downwards) whereas the leftwards arrows correspond to acting with $\overline{\bQ}_+$ (upwards) and $\overline{\bQ}_-$ (downwards).}
	\label{fig:quartet}
\end{figure}
%%%%%%

Returning to vertex algebras, if $\CV_{\rm matter}$ is a weak graded-unitary vertex algebra with a good Hamiltonian $\wh{\fg}_{-2h^\vee}$ action, then we can apply this Proposition to the graded subspaces of $C(\wh{\fg}_{-2h^\vee}, \fg, \CV_{\rm matter})$. While the adjoint operators $\overline{\bQ}_\pm$, although square-zero, are not derivations of the vertex algebra structure on $\CV$,  we can deduce the following weaker property that is formulated only in terms of the differentials $\bQ^\pm$.

%%%%%%
\begin{cor}[$\bQ^-\bQ^+$ lemma]\label{cor:ddc}
	Let $\CV_{\rm matter}$ be a weak graded-unitary vertex algebra with a good Hamiltonian $\wh{\fg}_{-2h^\vee}$ action. If $x \in C(\wh{\fg}_{-2h^\vee}, \fg, \CV_{\rm matter})$ is such that
	%%%%%%
    \[
		1)~\bQ^- x = 0 = \bQ^+ x \qquad \qquad 2)~x = \bQ^- x' ~\text{ or }~ x = \bQ^+ x'
	\]
    %%%%%%
	for some $x' \in C(\wh{\fg}_{-2h^\vee}, \fg, \CV_{\rm matter})$, then there exists $\wt{x} \in C(\wh{\fg}_{-2h^\vee}, \fg, \CV_{\rm matter})$ with $x = \bQ^- \bQ^+ \wt{x}$.
\end{cor}
%%%%%%

%%%%%%
\begin{rmk}
    The name $\bQ^- \bQ^+$ lemma is meant to parallel analogous results that hold for differential forms on compact K\"{a}hler manifolds, namely the $\diff\!\diff^c$ or $\pd \opd$ lemmas, \emph{cf.} Lemma 5.11 of \cite{DGMS}.
\end{rmk}
%%%%%%

%%%%%%
\begin{proof}
    We apply Proposition \ref{prop:quartet} to conclude that if $x$ satisfies the hypotheses then it belongs to a four-dimensional representation of the form in Figure \ref{fig:quartet}; we then take $\wt{x}$ to be the vector in this representation belonging to $\ker \overline{\bQ}_+ \cap \ker \overline{\bQ}_-$.
\end{proof}
%%%%%%
    %!TEX Root = ../formality.tex

%----------------------------------------------------------------%
\section{\label{sec:consequences}Consequences of graded unitarity}
%----------------------------------------------------------------%

The observations of the previous section have important consequences for the structure of the vertex algebra $C(\wh{\fg}_{-2h^\vee}, \fg, \CV_{\rm matter})$ and the relative semi-infinite cohomology $H^{\scriptstyle{\frac{\infty}{2}}+\bullet}(\wh{\fg}_{-2h^\vee}, \fg, \CV_{\rm matter})$. Our first result establishes that the cohomology inherits a graded-unitary structure from $\CV$. This proves, in particular, that the VOAs coming from Lagrangian $\CN=2$ SCFTs in four dimensions all possess a graded-unitary structure.

We then establish two properties that follow from the $\bQ^- \bQ^+$ lemma (Corollary \ref{cor:ddc}). The first is that the cohomology inherits an ${\rm USp}(2)$ automorphism from that of $C(\wh{\fg}_{-2h^\vee}, \fg, \CV_{\rm matter})$. The second property is that the DG vertex algebra $(C(\wh{\fg}_{-2h^\vee}, \fg, \CV_{\rm matter}), \bQ^-)$ is connected by a (chain of) quasi-isomorphisms of DG vertex algebras that to its cohomology (equipped with a trivial differential). We expect that in a suitable homotopy theory of vertex algebras this will amount to the formality for these DG vertex algebras.

For the third and final property, we use the full K\"{a}hler package (Theorem \ref{thm:BRSTKahlerpackage}) and consider reductive Lie algebras with multiple simple and abelian factors; we show that the relative semi-infinite cohomology for the full Lie algebra is isomorphic to an iterated relative semi-infinite cohomology for its simple factors. All M\"{o}bius vertex algebras appearing in this section should be assumed to be of CFT type.

%--------------------------------------------------------------------------------------%
\subsection{\label{sec:cohomology}Graded unitarity of relative semi-infinite cohomology}
%--------------------------------------------------------------------------------------%

The aim of this subsection is to prove the following Theorem:
%%%%%%
\begin{thm}\label{thm:gradedunitaryBRST}
    Let $\CV_{\rm matter}$ be a weak graded-unitary vertex algebra with a good Hamiltonian $\wh{\fg}_{-2h^\vee}$ action. Then the relative semi-infinite cohomology $H^{\scriptstyle{\frac{\infty}{2}} + \bullet}(\wh{\fg}_{-2h^\vee}, \fg, \CV_{\rm matter})$ is a weak graded-unitary vertex algebra. If $\CV_{\rm matter}$ is graded-unitary then so too is $H^{\scriptstyle{\frac{\infty}{2}} + \bullet}(\wh{\fg}_{-2h^\vee}, \fg, \CV_{\rm matter})$.
\end{thm}
%%%%%%

The immediate complication in deducing this result is that the conjugation $\rho$ does not commute with $\bQ^-$, instead satisfying
%%%%%%
\[
    \rho \circ \bQ^- = \ii \bQ^+ \circ \rho~.
\]
%%%%%%
We circumvent this issue by way of the following proposition.

%%%%%%
\begin{prop}\label{prop:symmquot}
    Let $\CV_{\rm matter}$ be a vertex algebra with a Hamiltonian $\wh{\fg}_{-2h^\vee}$ action such that $C(\wh{\fg}_{-2h^\vee}, \fg, \CV_{\rm matter})$ satisfies the $\bQ^- \bQ^+$ lemma (\emph{e.g.}, if $\CV_{\rm matter}$ is a weak graded-unitary vertex algebra and the action is good). Then there is an isomorphism of vertex algebras
    %%%%%%
    \[
		H^{\scriptstyle{\frac{\infty}{2}} + \bullet}(\wh{\fg}_{-2h^\vee}, \fg, \CV_{\rm matter}) \coloneqq \frac{\ker \bQ^-}{{\rm im} \bQ^-} \cong \frac{\ker \bQ^- \cap \ker \bQ^+}{{\rm im} \bQ^-\bQ^+}~.
	\]
    %%%%%%
\end{prop}
%%%%%%

%%%%%%
\begin{proof}
    This result follows from the $\bQ^- \bQ^+$ lemma by showing that the inclusion
    %%%%%%
    \[
        \ker \bQ^- \cap \ker \bQ^+ \hookrightarrow \ker \bQ^-
    \]
    %%%%%%
    induces the desired isomorphism. We start with injectivity; choose an element $x$ with $\bQ^\pm x = 0$ representing a class in the source and suppose its image vanishes, \emph{i.e.} $x$ is $\bQ^-$-exact. The $\bQ^- \bQ^+$ lemma then implies $x = \bQ^- \bQ^+ \wt{x}$, hence $x$ represents the zero class and the induced map is injective. To show it is surjective, choose $[x]$ in $\bQ^-$-cohomology represented by the $\bQ^-$-closed element $x$. The $\bQ^- \bQ^+$ lemma implies the existence of $\wt{x}$ with $\bQ^+ x = \bQ^- \bQ^+ \wt{x}$ and hence $x' = x + \bQ^- \wt{x}$ satisfies $\bQ^\pm x' = 0$ and $[x'] = [x]$, therefore the induced map is surjective.
\end{proof}
%%%%%%

We prove Theorem \ref{thm:gradedunitaryBRST} by equipping this alternative quotient with the structure of a weak graded-unitary vertex algebra. For the sake of brevity we will denote this quotient by $\wh{\CV}$ and write
%%%%%%
\[
    [\![x]\!] \coloneqq x + {\rm im}\bQ^- \bQ^+~
\]
%%%%%%
for $x \in \ker \bQ^- \cap \ker\bQ^+$. There is an obvious half-integral filtration on $\wh{\CV}$, which is good if the filtration on $\CV_{\rm matter}$ is good:
%%%%%%
\[
    \wh{\mathfrak{R}}_R \wh{\CV} = \left\{[\![x]\!] ~ \big| ~ x \in \ker \bQ^- \cap \ker \bQ^+\cap\mathfrak{R}_R C(\wh{\fg}_{-2h^\vee}, \fg, \CV_{\rm matter})\right\}~.
\]
%%%%%%
The invariant bilinear form on $C(\wh{\fg}_{-2h^\vee}, \fg, \CV_{\rm matter})$ induces an invariant bilinear form on $\wh{\CV}$, which we also denote by $(-,-)$, where
%%%%%%
\[
    ([\![x]\!], [\![y]\!]) = (x,y)~.
\]
%%%%%%

We then have the following lemma.

%%%%%%
\begin{lem}
    $\wh{\mathfrak{R}}_\bullet$ is a nondegenerate half-integral filtration on $\wh{\CV}$. Moreover, it is good if the filtration $\mathfrak{R}_\bullet$ is good.
\end{lem}
%%%%%%

%%%%%%
\begin{proof}
    That $\wh{\mathfrak{R}}_\bullet$ defines a half-integral filtration on $\wh{\CV}$ is immediate from the fact that $\mathfrak{R}_\bullet$ is a filtration, and similarly that it is good if $\mathfrak{R}_\bullet$ is good. We are thus left to establish nondegeneracy.

    Choose $[\![x]\!] \in \wh{\CV}$ nonzero. We can represent such a coset by its harmonic representative $x'$ and get
    %%%%%%
    \[
        0 \neq \langle x' | x' \rangle = ((\sigma \circ \rho)(x'), x')~.
    \]
    %%%%%%
    As $y \coloneqq (\sigma \circ \rho)(x')$ is also harmonic, and thus belongs to $\ker \bQ^+ \cap \ker \bQ^-$, we conclude that the pairing $([\![y]\!], [\![x']\!]) = ([\![y]\!], [\![x]\!])$ is nonzero so that the pairing is nondegenerate.
\end{proof}
%%%%%%

%%%%%%
\begin{rmk}
    The nondegeneracy of $\wh{\mathfrak{R}}_\bullet$ ensures that it can be refined to a grading on the underlying vector space. The graded subspaces are given by
    \[
        \wh{\CV}_{h, R} = \left\{ [\![x]\!] ~ \big| ~ x = x' + \bQ^- \bQ^+ \wt{x}~, \quad x' \in C(\wh{\fg}_{-2h^\vee}, \fg, \CV_{\rm matter})_{h, R} \right\}~.
    \]
    Of course, if $x'$ is annihilated by $\bQ^+$ and $\bQ^-$ and has definite $\mathfrak{R}$-degree then it is annihilated by all four of $Q^\pm$, $S^\pm$, \emph{i.e.}, it is harmonic. In other words, $[\![x]\!]$ belongs to the $R^{\rm th}$ graded component of $\wh{\CV}$ if and only if it can be represented by a harmonic element in the $R^{\rm th}$ graded component of $C(\wh{\fg}_{-2h^\vee}, \fg, \CV_{\rm matter})$.
\end{rmk}
%%%%%%

As $\rho$ preserves both $\ker \bQ^- \cap \ker \bQ^+$ and ${\rm im}\bQ^- \bQ^+$, we also have a natural quaternionic structure,
%%%%%%
\[
    \wh{\rho} : [\![x]\!] \mapsto [\![\rho(x)]\!]~.
\]
%%%%%%
That this defines a quaternionic structure on $(\wh{\CV}, (-,-), \wh{\mathfrak{R}}_\bullet)$ follows immediately from the fact that $\rho$ defines a quaternionic structure on $(C(\wh{\fg}_{-2h^\vee}, \fg, \CV_{\rm matter}), \mathfrak{R}_\bullet)$. Theorem \ref{thm:gradedunitaryBRST} then follows from the following lemma.

%%%%%%
\begin{lem}
    Taking $(\wh{\CV},(-,-),\wh{\mathfrak{R}}_\bullet,\wh{\rho})$ as above, the corresponding sesquilinear form defined in Corollary \ref{cor:Hermitianform} is positive-definite.
\end{lem}
%%%%%%

%%%%%%
\begin{proof}
    Choose $[\![x]\!] \in \wh{\CV}_{h}$ and let
    %%%%%%
    \[
        [\![x]\!] = \sum_{R} [\![x_R]\!]~, \qquad x_R = x_R' + \bQ^- \bQ^+ \wt{x}_R~, \quad x'_R \in C(\wh{\fg}_{-2h^\vee}, \fg, \CV_{\rm matter})_{h, R}
    \]
    %%%%%%
    be its decomposition into graded components. Then we have
    %%%%%%
    \[
    \begin{aligned}
        \langle [\![x]\!] | [\![x]\!] \rangle & = ((\wh{\sigma} \circ \wh{\rho})([\![x]\!]), [\![x]\!]) = \sum_{R} ((\wh{\sigma} \circ \wh{\rho})([\![x_R']\!]), [\![x_R']\!])\\
        & = \sum_{R} ((\sigma \circ \rho)(x'_R), x'_R) = \sum_R \langle x'_R | x'_R \rangle \geqslant 0~,
    \end{aligned}
    \]
    %%%%%%
    as desired.
\end{proof}
%%%%%%

%-----------------------------------------------------------------------------------%
\subsection{\label{sec:automorphisms}Outer automorphisms from the $\bQ^-\bQ^+$ lemma}
%-----------------------------------------------------------------------------------%

We saw in Section \ref{sec:BRSTreview} that the vertex algebra underlying the relative semi-infinite cohomology admits has an interesting outer automorphism group: there is an ${\rm USp}(2)$ automorphism that rotates the symplectic fermions $\eta^{+,A}$, $\eta^{-,A}$ into one another. As this action rotates the differential $\bQ^-$ as a doublet with $\bQ^+$, the cohomology does not necessarily inherit this automorphism: only the torus acting on $\eta^{\pm, A}$ with weight $\mp 1$ preserves the kernel and image of $\bQ^-$. Indeed, assuming $\CV_{\rm matter}$ is concentrated in degree zero, the weights of this ${\rm U}(1)$ action are precisely the cohomological grading on $\bQ^-$ cohomology.

Nonetheless, the cohomologies of $\bQ^-$ and $\bQ^+$ are isomorphic to one another, so there is a chance that some vestige of the automorphism group will give rise to an action on $\bQ^-$ cohomology. We now show that this is indeed the case whenever the $\bQ^-\bQ^+$ lemma holds.

%%%%%%
\begin{prop}\label{prop:automorphisms}
    Let $\CV_{\rm matter}$ be a vertex algebra with a Hamiltonian $\wh{\fg}_{-2h^\vee}$ action such that $C(\wh{\fg}_{-2h^\vee}, \fg, \CV_{\rm matter})$ satisfies the $\bQ^- \bQ^+$ lemma (\emph{e.g.}, if $\CV_{\rm matter}$ is a weak graded-unitary vertex algebra and the action is good). Then the ${\rm U}(1)$ action on relative semi-infinite cohomology with weight spaces $H^{\scriptstyle{\frac{\infty}{2}}+d}(\wh{\fg}_{-2h^\vee}, \fg, \CV_{\rm matter})$ extends to an action of ${\rm USp}(2)$.
\end{prop}
%%%%%%

%%%%%%
\begin{rmk}
    In light of the formal analogy with K\"{a}hler geometry described in Section \ref{sec:Kahler}, the fact that $\fsl(2)$ acts on the cohomology of, say, $\bQ^-$ follows from the K\"{a}hler-like identities in Eqs. \eqref{eq:Lefschetz1} and \eqref{eq:Lefschetz2} together with the Hodge decomposition. That said, this does not guarantee that it is acts in a way compatible with the vertex algebra structure, \emph{i.e.}, by vertex algebra automorphisms.
\end{rmk}
%%%%%%

%%%%%%
\begin{proof}
    The assertion follows from Proposition \ref{prop:symmquot}, which provides a manifestly ${\rm USp}(2)$-invariant description of $H^{\scriptstyle{\frac{\infty}{2}} + \bullet}(\wh{\fg}_{-2h^\vee}, \fg, \CV_{\rm matter}) \cong \wh{\CV}$. 
\end{proof}
%%%%%%

%%%%%%
\begin{physrmk}
    The appearance of an ${\rm USp}(2)$ outer automorphism on relative semi-infinite cohomology is quite surprising from the perspective of four-dimensional physics. The cohomological grading is interpreted physically as weights of the ${\rm U}(1)_r$ $R$-symmetry of the physical theory and we do not know of a physical reason why this should have a non-abelian enhancement, though upon reduction to three dimensions it seems plausible that the enhancement could be related to the ${\rm SU}(2)_C$ enhancement of ${\rm U}(1)_r$ (\emph{cf.} footnote \ref{footnote:outer_speculations}).
\end{physrmk}
%%%%%%

%----------------------------------------------------------------------%
\subsection{\label{sec:formality}Formality from the $\bQ^- \bQ^+$ lemma}
%----------------------------------------------------------------------%

As in the analysis of the de Rham cohomology of compact K\"{a}hler manifolds in \cite{DGMS}, formality of the relative semi-infinite cohomology complex follows as a consequence of the $\bQ^- \bQ^+$ lemma.

%%%%%%
\begin{thm}\label{thm:formality}
    Let $\CV_{\rm matter}$ be a vertex algebra equipped with a Hamiltonian $\wh{\fg}_{-2h^\vee}$ action such that $\CV \coloneqq C(\wh{\fg}_{-2h^\vee}, \fg, \CV_{\rm matter})$ satisfies the $\bQ^- \bQ^+$-lemma (\emph{e.g.}, if $\CV_{\rm matter}$ is a weak graded-unitary vertex algebra and the action is good). The maps
	%%%%%%
    \[
        (\CV, \bQ^-) \overset{i}{\longleftarrow} (\ker \bQ^+, \bQ^-) \overset{\pi}{\longrightarrow} (H(\CV, \bQ^+), \bQ^-{}^*)~,
    \]
    %%%%%%
    where $i$ is the inclusion of the subcomplex $(\ker \bQ^+, \bQ^-)$ and $\pi$ is the projection onto $\bQ^+$-cohomology, are quasi-isomorphisms.
\end{thm}
%%%%%%

%%%%%%
\begin{rmk}
    In the statement of this proof and below, we use $m^*$ to denote the map on cohomology induced by a linear map $m$. This should not be confused with the adjoint with respect to the nondegenerate pairing $(-,-)$ as used in, \emph{e.g.}, Proposition \ref{prop:shortening}.
\end{rmk}
%%%%%%

%%%%%%
\begin{rmk}
    The maps appearing in Theorem \ref{thm:formality} are morphisms of DG vertex algebras. This follows from the fact that $\bQ^-$ and $\bQ^+$ (anti-)commute, that $\ker \bQ^+$ is a vertex subalgebra of $\CV$, and that $H(\CV, \bQ^+)$ is a quotient of $\ker \bQ^+$ by a vertex algebra ideal.
\end{rmk}
%%%%%%

%%%%%%
\begin{proof}
    The proof closely follows the ``${\rm d}^c$-Diagram Method'' of \cite{DGMS}, but we present the details here for completeness.
	
    The induced differential $\bQ^-{}^*$ acts trivially on $H(\CV, \bQ^+)$, because if a $\bQ^+$-closed $x$ represents such a cohomology class, we can apply the $\bQ^- \bQ^+$ lemma to $\bQ^- x$ to find an operator $\wt{x}$ with $\bQ^- x = \bQ^- \bQ^+ \wt{x} = \bQ^+ (-\bQ^- \wt{x})$, and hence $\bQ^- x$ vanishes in $H(\CV, \bQ^+)$. In particular, one may replace the right-most vertex algebra with $(H(\CV, \bQ^+), 0)$ without any loss.
	
    We now show that $i$ is a quasi-isomorphism. Choose $[x] \in H(\CV, \bQ^-)$ and consider $\bQ^+ x$. The operator $\bQ^+ x$ satisfies the hypothesis of the $\bQ^- \bQ^+$ lemma, so there exists $\wt{x}$ with $\bQ^+ \CO = \bQ^- \bQ^+ \wt{x}$. We find that $x' = x + \bQ^- \wt{x}$ is $\bQ^+$-closed and $\bQ^-$-cohomologous to $\CO$, thus $i^*$ is surjective. To show that $i^*$ is an injection, we consider a vector $x$ satisfying $\bQ^+ x = 0$ and $x = \bQ^- x'$ for some $x' \in \CV$; the key distinction being that $x'$ need not belong to $\ker \bQ^+$. However, since $x$ satisfies the hypotheses of the $\bQ^- \bQ^+$ lemma, we can write $x = \bQ^- \bQ^+ \wt{x}$. As $\bQ^+(\bQ^+ \wt{x}) = 0$ we conclude that $x$ is necessarily $\bQ^-$-exact in $\ker \bQ^+$.
	
    Finally, we prove that $\pi$ is a quasi-isomorphism. To check surjectivity, we suppose $x$ represents an element of $H(\CV, \bQ^+)$; we must find a $\bQ^-$-closed operator that is $\bQ^+$-cohomologous to $x$. It follows that $\bQ^- x$ satisfies the hypotheses of the $\bQ^- \bQ^+$ lemma, hence there is a vector $\wt{x}$ such that $\bQ^- x = \bQ^- \bQ^+ \wt{x}$. In particular, $x' = x + \bQ^+ \wt{x}$ is $\bQ^-$-closed and $\bQ^+$-cohomologous to $x$. To see that $\pi^*$ is injective, we consider a $\bQ^-$-closed element $x$ of $\ker \bQ^+$ and suppose $x = \bQ^+ \wh{x}$ for some $\wh{x} \in \CV$, \emph{i.e.}, it maps to zero in $H(\CV, \bQ^+)$. We see that $x$ satisfies the $\bQ^- \bQ^+$ lemma, and hence there exists $\wt{x}$ with $x = \bQ^- \bQ^+ \wt{x}$, therefore $x$ is $\bQ^-$-exact in $\ker \bQ^+$.
\end{proof}
%%%%%%

%%%%%%
\begin{rmk}
    There is a second proof of formality given in \cite{DGMS} by way of the ``Principle of Two Types'', which establishes the uniform vanishing of higher Massey products for compact K\"{a}hler manifolds. This method is not available to us as there is at present no understanding of higher Massey-like products for vertex algebras.
\end{rmk}
%%%%%%

%----------------------------------------------------------------%
\subsection{\label{sec:iterated}Iterated semi-infinite cohomology}
%----------------------------------------------------------------%

For this last subsection, we will assume that $\CV_{\rm matter}$ satisfies the hypotheses of Theorem \ref{thm:BRSTKahlerpackage}. We further assume the (compact) Lie algebra $\fg_c$ can be written as the direct sum of two such Lie algebras $\fg_c = \fg_{c,1} \oplus \fg_{c,2}$ and hence can write $\fg = \fg_1 \oplus \fg_2$, where $\fg_1$, $\fg_2$ are the complexifications of $\fg_{c,1}$, $\fg_{c,2}$. For the sake of brevity we set $\CV = C(\wh{\fg}_{-2h^\vee}, \fg, \CV_{\rm matter})$.

With these assumptions, we can decompose the $\mathfrak{sl}(2)$ triple $\{\Pi, L, \Lambda\}$
%%%%%%
\[
	\Pi = \Pi_1 + \Pi_2~, \qquad L = L_1 + L_2~, \qquad \Lambda = \Lambda_1 + \Lambda_2~,
\]
%%%%%%
as well as the differentials $\bQ^\alpha$, their adjoints $\overline{\bQ}_\alpha$, and the Laplacian $\Delta$
%%%%%%
\[
	\bQ^\alpha = \bQ^\alpha_1 + \bQ^\alpha_2~, \qquad \overline{\bQ}_\alpha = \overline{\bQ}_{1,\alpha} + \overline{\bQ}_{2,\alpha}~, \qquad \Delta = \Delta_1 + \Delta_2~.
\]
%%%%%%
The operators $\Pi_a$, $L_a$, $\Lambda_a$, $\bQ^\alpha_a$, $\overline{\bQ}_{a, \alpha} = (\bQ^\alpha_a)^\dagger$, and $\Delta_a$ separately satisfy the K\"{a}hler identities in Eqs. \eqref{eq:Hodge}, \eqref{eq:Lefschetz1}, and \eqref{eq:Lefschetz2}. Moreover, there are separate Hodge decompositions for each $a = 1,2$ and hence separate K\"{a}hler packages with respect to $\fg_1$ and to $\fg_2$ on $\wh{\CV}$.
%%%%%%
\begin{prop}\label{prop:iteratedHodgerelations}
    The operators $\Pi_a$, $L_a$, $\Lambda_a$, $\bQ^\alpha_a$, $\overline{\bQ}_{a, \alpha} = (\bQ^\alpha_a)^\dagger$, and $\Delta_a$ and corresponding Hodge decompositions equip $\CV$ with two separate K\"{a}hler packages.
\end{prop}
%%%%%%

As an immediate consequence, $\CV$ satisfies the $\bQ^-\bQ^+$ lemma for each of these packages.

%%%%%%
\begin{cor}\label{cor:doubleddc}
	Let $x \in \CV$ be such that
	%%%%%%
    \[
		1)~\bQ^-_1 x = 0 = \bQ^+_1 x \qquad \qquad 2)~x = \bQ^-_1 x' ~\text{ or }~ x = \bQ^+_1 x'
	\]
    %%%%%%
	for some $x' \in \CV$, then there exists $\wt{x} \in \CV$ with $x = \bQ^-_1 \bQ^+_1 \wt{x}$. Similarly, if $y$ is such that
	%%%%%%
    \[
		1)~\bQ^-_2 y = 0 = \bQ^+_2 y \qquad \qquad 2)~y = \bQ^-_2 y' ~\text{ or }~ y = \bQ^+_2 y'
	\]
    %%%%%%
	for some $y' \in \CV$, then there exists $\wt{y} \in \CV$ with $y = \bQ^-_2 \bQ^+_2 \wt{y}$.
\end{cor}
%%%%%%

The following is an immediate corollary of Theorem \ref{thm:gradedunitaryBRST} together with Corollary \ref{cor:ddc}.

%%%%%%
\begin{cor}\label{cor:iteratedddc}
    Let $[x] \in H(\CV, \bQ^+_1)$ be such that
    %%%%%%
    \[
        1) \bQ^-_2{}^* [x] = 0 = \bQ^+_2{}^* [x] \qquad 2) [x] = \bQ^-_2{}^*[x'] ~\text{ or }~ [x] = \bQ^+_2{}^*[x']
    \]
    %%%%%%
    for some $[x'] \in H(\CV, \bQ^+_1)$, then there exists $[\wt{x}] \in H(\CV, \bQ^+_1)$ with $[x] = \bQ^-_2{}^* \bQ^+_2{}^* [\wt{x}]$.
\end{cor}
%%%%%%

With these preliminary results in hand, the following is a simple extension of Theorem \ref{thm:formality}.

%%%%%%
\begin{prop}\label{prop:iteratedformality}
    The following is a diagram of morphisms of DG vertex algebras all of which are quasi-isomorphisms:
	%%%%%%
    \[
	%%%%%%
    \begin{tikzpicture}
		\draw (0,0) node {$(\CV, \bQ^-)$};
		\draw (5,0) node {$(H(\CV, \bQ_1^+), \bQ^-_2{}^*)$};
		\draw (10,0) node {$(H(H(\CV, \bQ^+_1), \bQ^+_2{}^*), 0)$};
		\draw (2.5,1.5) node {$(\ker \bQ^+_1, \bQ^-)$};
		\draw (7.5,1.5) node {$(\ker \bQ^+_2{}^*, \bQ^-_2{}^*)$};
		
		\draw[right hook->] (2, 1.25) -- (0.6,0.4);
		\draw[->>] (3,1.25) -- (4.4,0.4);
		\draw[right hook->] (7, 1.25) -- (5.6,0.4);
		\draw[->>] (8,1.25) -- (9.4,0.4);
	\end{tikzpicture}
    %%%%%%
	\]
    %%%%%%
\end{prop}
%%%%%%

%%%%%%
\begin{proof}
    We first comment on the form of the diagram. In the proof of Theorem \ref{thm:formality} we saw that the action of $\bQ^-{}^*$ on $H(\CV, \bQ^+)$ was trivial by applying the $\bQ^- \bQ^+$ lemma. The same is true here: the action of $\bQ^-{}^*$ on $H(\CV, \bQ^+_1)$ is equal to $\bQ^-_2{}^*$ due to Corollary \ref{cor:doubleddc} and the action of $\bQ^-_2{}^*{}^*$ on $H(H(\CV, \bQ^+_1), \bQ^+_2{}^*)$ is trivial due to Corollary \ref{cor:iteratedddc}. We have made both of these replacements in the above diagram. With this in mind, it is easy to see that the maps appearing in the diagram are indeed vertex algebra morphisms. The leftwards arrows are inclusions of kernels of derivations, hence inclusions of subalgebras, and the rightwards arrows are quotients by images of differentials, hence quotients by ideals, so each of these maps are indeed morphisms of vertex algebras. We now show that each of the maps appearing here are quasi-isomorphisms.
	
	The fact that the right two morphisms are quasi-isomorphisms is identical to the proof of Theorem \ref{thm:formality} by repeatedly applying Corollary \ref{cor:iteratedddc}. To see that the left two morphisms are quasi-isomorphisms, we view each of the complexes appearing as bi-complexes with horizontal and vertical differentials $\bQ^-_1$, $\bQ^-_2$ (or the corresponding induced maps). By restricting to the subspaces of fixed conformal dimension, which are finite dimensional, Theorem \ref{thm:formality} implies that these maps are quasi-isomorphisms on each row and hence induce quasi-isomorphisms on the total complex.
\end{proof}
%%%%%%

An immediate corollary of this result is that the relative semi-infinite $\wh{\fg}_{-2h^\vee}$-cohomology is isomorphic to the iterated cohomology for its direct summands.
%%%%%%
\begin{cor}\label{cor:iteratedcohomology}
    Let $\CV_{\rm matter}$ be a weak graded-unitary vertex algebra with a good Hamiltonian $\wh{\fg}_{-2h^\vee}$ action and let $\fg = \bigoplus_{k=1}^N \fg_k$ be a direct sum decomposition of $\fg$. Then there is an isomorphism of vertex algebras
	%%%%%%
    \[
		H(\wh{\fg}_{-2h^\vee}, \fg, \CV_{\rm matter}) \cong H(\wh{\fg}_{N,-2h_N^\vee}, \fg_N, \dots H(\wh{\fg}_{-2h^\vee_1}, \fg_1, \CV_{\rm matter})\dots)~.
	\]
    %%%%%%
\end{cor}
%%%%%%

%%%%%%
\begin{rmk}
    Combined with Proposition \ref{prop:automorphisms}, we see that there are as many ${\rm USp}(2)$ automorphisms as the number of simple factors plus the rank of the radical of $\fg$, although some may act trivially.
\end{rmk}
%%%%%%

%%%%%%
\begin{physrmk}
    From the perspective of four-dimensional physics, this result reflects the fact that the gauge coupling is exactly marginal and hence it does not whether one arrives at a point with nonzero gauge couplings all at once or in stages. For three-dimensional boundary VOAs, this result is more physically nontrivial, as gauging triggers a renormalization group (RG) flow. This result implies that at the level of the boundary VOA, the RG flows associated to gauging different direct summands in the gauge algebra commute.
\end{physrmk}
%%%%%%
    %!TEX Root = ../formality.tex

%---------------------------------------------------------------------------------------%
\section{\label{sec:PVAandHLCR}Poisson vertex algebras and Hall--Littlewood chiral rings}
%---------------------------------------------------------------------------------------%

We now turn to the structural properties of certain algebraic objects that can be extracted from the graded-unitary vertex algebras that arise from four-dimensional $\CN=2$ SCFTs. The first object we study is the Poisson vertex algebra (PVA) obtained by taking the associated graded with respect to the (good) $\mathfrak{R}$-filtration underlying a graded-unitary vertex algebra. From the perspective of four-dimensional physics, this PVA is identified with the algebra of local operators in the holomorphic-topological twist (also called the Kapustin twist) of the SCFT \cite{Kapustin:2006hi, Oh:2019bgz,Jeong:2019pzg}. As we will see, this PVA inherits a Hermitian inner product from its parent graded-unitary vertex algebra. Moreover, we find that the PVAs coming from the vertex algebras underlying relative semi-infinite vertex algebras are equipped with a K\"{a}hler package and hence share many of the properties of their parent graded-unitary vertex algebras, such as formality.

We conclude with an application to finite-type Poisson algebraic geometry. Namely, we instantiate the additional $\Z$ grading mentioned in Remark \ref{rmk:Zgrading} and Physical Remark \ref{physrmk:Zgrading} to extract from the associated graded PVA a Poisson algebra called the Hall--Littlewood chiral ring. For the DG PVAs coming from the DG vertex algebras underlying relative semi-infinite cohomology, this Poisson algebra can be identified with a derived Poisson reduction; we show that formality of the PVA implies formality of this derived symplectic quotient. All M\"{o}bius vertex algebras appearing in this section are assumed to be of CFT type.

%------------------------------------------------------------------------%
\subsection{\label{sec:PVA}K\"{a}hler packages for associated graded PVAs}
%------------------------------------------------------------------------%

As shown by Li \cite[Prop. 4.2]{Li2004}, see also Proposition \ref{prop:assocgradedPVA}, if $\CV$ is a vertex algebra equipped with a (half-integral) good filtration $\CF_\bullet$ (\emph{i.e.}, $\CV$ is a (half-integer-)filtered vertex algebra), then the associated graded ${\rm gr}_\CF \CV$ is naturally a PVA. When $\CV$ is a M\"{o}bius vertex algebra equipped with a nondegenerate invariant bilinear form $(-,-)$ for which this good filtration is nondegenerate, we saw in Section \ref{sec:unitary} that the filtration can be refined to a grading on the underlying vector space by the Gram-Schmidt procedure. In this situation, the PVA structure can be made rather explicit using this identification of the graded components of ${\rm gr}_\CF \CV$. If $x \in \CV_p$ is a homogeneous vector of degree $p$ we can decompose its modes $x_n$ into homogeneous pieces as
%%%%%%
\[
    x_n = x_n^{[p]} + \dots~,
\]
%%%%%%
where $\dots$ includes terms of degree less than $p$. The state-field correspondence of the associated graded PVA then assigns to $x$ the field
%%%%%%
\[
    Y_+(x,z) = \sum_{n \in \Z} x_n^{[p]} z^{-n-1}~.
\]
%%%%%%
The goodness of the filtration ensures that $x_n^{[p]} = 0$ for $n \geqslant 0$, so this field belongs to ${\rm End}(\CV)[\![z]\!]$. Similarly, the vertex Lie algebra structure on ${\rm gr}_\CF \CV$ takes the form
%%%%%%
\[
    Y_-(x, z) = \sum_{n \geqslant 0} x^{[p-1]}_n z^{-n-1}~.
\]
%%%%%%
Note that if $\CV$ is a graded-unitary vertex algebra then the associated graded PVA ${\rm gr}_\mathfrak{R} \CV$ inherits a Hermitian inner product via the isomorphism of the underlying vector spaces.

%%%%%%
\begin{lem}\label{lem:HermitianPVA}
    Let $(\CV, (-,-), \mathfrak{R}_\bullet, \rho)$ be a graded-unitary vertex algebra, then the associated graded ${\rm gr}_\mathfrak{R} \CV$ is naturally equipped with a Hermitian inner product. 
\end{lem}
%%%%%%

%%%%%%
\begin{rmk}
    Although the PVA is equipped with an inner product, it is \emph{not} invariant with respect to the vertex Lie algebra structure. As a result, it is less natural than the inner product appearing in the parent graded-unitary vertex algebra.
\end{rmk}
%%%%%%

Now suppose $\CV_{\rm matter}$ is a graded-unitary vertex algebra with a good Hamiltonian $\wh{\fg}_{-2h^\vee}$ action. Set $\CV = C(\wh{\fg}_{-2h^\vee}, \fg, \CV_{\rm matter})$, which is naturally a graded-unitary vertex algebra equipped with a pair of commuting differentials $\bQ^\pm$ rotated into each other by an ${\rm USp}(2)$ automorphism; indeed, it is equipped with a full K\"{a}hler package, \emph{cf.} Theorem \ref{thm:BRSTKahlerpackage}. The associated graded PVA ${\rm gr}_{\mathfrak{R}} \CV$ of $\CV$ is thus equipped with a pair of commuting differentials ${\rm gr} \bQ^\pm$; we can identify the action of these differentials with the terms in $\bQ^\pm$ of degree $\frac{1}{2}$, \emph{i.e.} ${\rm gr} \bQ^\pm = Q^\pm$. Moreover, the action of ${\rm USp}(2)$ coming from the ``Lefschetz'' $\fsl(2)$ triple $\{\Pi, L, \Lambda\}$ induces an action on ${\rm gr}_\mathfrak{R} \CV$ that rotates $Q^\pm$ into each other. We already identified the adjoints of these differentials with respect to the Hermitian inner product on $\CV$ as coming from the terms of degree $-\frac{1}{2}$ of the other differential: $\overline{Q}_\pm \coloneqq (Q^\pm)^\dagger = \mp S^\mp$. Together with the Laplacian $\Delta$, we see that these maps satisfy
%%%%%%
\begin{equation}\label{eq:PVAKahler}
    [Q^\alpha, Q^\beta] = 0~, \qquad [\overline{Q}_\alpha, \overline{Q}_\beta] = 0~, \qquad [Q^\alpha, \overline{Q}_\beta] = \tfrac{1}{2} \delta^\alpha_\beta \Delta~.
\end{equation}
%%%%%%
We can summarize these observations as a PVA analogue of Theorem \ref{thm:BRSTKahlerpackage}:

%%%%%%
\begin{prop}\label{prop:PVAKahler}
    Let $\CV_{\rm matter}$ be a graded-unitary vertex algebra with a good Hamiltonian $\wh{\fg}_{-2h^\vee}$ action. Then the associated graded PVA ${\rm gr}_\mathfrak{R} C(\wh{\fg}_{-2h^\vee}, \fg, \CV_{\rm matter})$ is equipped with a K\"{a}hler package.
\end{prop}
%%%%%%

As with graded-unitary vertex algebras, the most important consequence of this K\"{a}hler package is the fact that the associated graded PVA satisfies a $Q^+Q^-$ lemma. In fact, the K\"{a}hler package described above is not necessary to establish this result; it follows from identifying the underlying vector spaces together with the fact that $C(\wh{\fg}_{-2h^\vee}, \fg, \CV_{\rm matter})$ satisfies the $Q^+ Q^-$ lemma.%
\footnote{That $C(\wh{\fg}_{-2h^\vee}, \fg, \CV_{\rm matter})$ satisfies the $Q^+Q^-$ lemma follows by repeating the analysis in Section \ref{sec:recombination} with $\bQ^\pm$ (resp $\overline{\bQ}_\pm$) replaced by $Q^\pm$ (resp. $\overline{Q}_\pm$). This is analogous to the fact that the $\diff\!\diff^c$ lemma and $\pd \opd$ lemma are equivalent for compact K\"{a}hler manifolds.} %
%

%%%%%%
\begin{cor}\label{cor:ddcPVA}
    Let $\CV_{\rm matter}$ be a graded-unitary vertex algebra with a good Hamiltonian $\wh{\fg}_{-2h^\vee}$ action and set $\CV = (\CV_{\rm matter} \otimes {\rm Sf}[\C^2\otimes\fg])^G$. Then the associated graded PVA ${\rm gr}_\mathfrak{R} \CV$ satisfies the $Q^+ Q^-$ lemma.
\end{cor}
%%%%%%

With the $Q^+ Q^-$ lemma in hand, all the consequences described in Section \ref{sec:consequences} carry over to give structural properties of ${\rm gr}_\mathfrak{R} \CV$. For example, the DG PVA $({\rm gr}_\mathfrak{R} \CV, Q^-)$ is formal, \emph{i.e.}, it is quasi-isomorphic (as a PVA) to its cohomology, and its cohomology has admits an action of ${\rm USp}(2)$ extending its cohomological grading.

%--------------------------------------------------------------------%
\subsection{\label{sec:HL}Formality for Hall--Littlewood chiral rings}
%--------------------------------------------------------------------%

We now describe a special type of graded-unitary vertex algebra that more closely models those coming from the construction \cite{Beem:2013sza}. Recall that when the Grassmann parity the underlying graded-unitary vertex algebra is induced from a $\Z$-grading as in Remark \ref{rmk:Zgrading} and Physical Remark \ref{physrmk:Zgrading}, the vector space underlying the vertex algebra is triply graded,
%%%%%%
\[
    \CV = \bigoplus_{\substack{h \in \scriptstyle{\frac{1}{2}}\N\\ R \in \scriptstyle{\frac{1}{2}}\N\\d \in \Z\\ }} \CV_{h,R,d}~,
\]
%%%%%%
where $h$ is the conformal weight grading, $R$ is the grading that induces the $\mathfrak{R}$-filtration, and $d$ is the weights of the internal $\Z$ grading that induces the $\Z_2$ action of the graded-unitary structure. Importantly, if this vertex algebra comes from a unitary four-dimensional $\CN=2$ SCFT then these gradings satisfy a \emph{BPS bound}:
%%%%%%
\[
	h - R - |r| \geqslant 0 \quad \rightsquigarrow \quad h \geqslant R + \tfrac{1}{2}|d|~.
\]
%%%%%%
This requirement is part of the notion of a \emph{filtration of four-dimensional type} given in Def. 2.3 of \cite{ArabiArdehali:2025fad}; we will not need this full structure in the following. The operators that saturate the BPS bound are preserved by yet more superconformal symmetry than Schur operators and are called Hall--Littlewood operators if $r \geqslant 0$, anti-Hall--Littlewood operators if $r \leqslant 0$, and Higgs branch chiral operators if $r = 0$.

%%%%%%
\begin{dfn}\label{dfn:HLOps}
    A graded-unitary vertex algebra $(\CV, (-,-), \mathfrak{R}_\bullet, \rho)$ is said to \defterm{satisfy the BPS bound} if
    %%%%%%
    \[
        \CV_{h,R,d} = 0 \quad {\rm for} \quad h < R + \tfrac{1}{2}|d|~.
    \]
    %%%%%%
    The vectors in $\CV_{R-\scriptstyle{\frac{1}{2}}d, R, d}$ ($d \leqslant 0$) are referred to as \defterm{Hall--Littlewood operators}, whereas the vectors in $\CV_{R+\scriptstyle{\frac{1}{2}}d, R, d}$ ($d \geqslant 0$) are referred to \defterm{anti-Hall--Littlewood operators}. The vectors in $\CV_{R, R, 0}$ are referred to as \defterm{Higgs branch operators}.
\end{dfn}
%%%%%%

%%%%%%
\begin{rmk}
    The examples of graded-unitary vertex algebras given in Section \ref{sec:examples} can be given internal $\Z$ gradings that satisfy the BPS bound. For the symplectic boson VOA ${\rm Sb}[V]$, the internal grading is trivial: the generators have $h = \frac{1}{2}$ and $R = \frac{1}{2}$ and hence must be given $d = 0$. The Hall--Littlewood, anti-Hall--Littlewood, and Higgs branch operators all coincide and are given by the $q^a(z)$ and normally-ordered products thereof. For the symplectic fermion VOA ${\rm Sf}[\C^2 \otimes U]$, the internal grading satisfies the BPS bound if we give the generator $\eta^{\pm, A}$ from $U_\pm$ degree $d = \pm 1$, matching the physically-relevant grading. The Hall--Littlewood (resp. anti-Hall--Littlewood) operators are given by normally-ordered products of $\eta^{-,A}(z)$ (resp. $\eta^{+,A}(z)$); there are no Higgs branch operators. Note that if $\CV_1$, $\CV_2$ are graded-unitary vertex algebras whose internal $\Z$ gradings satisfy the BPS bound, then so does the diagonal grading on $\CV_1 \otimes \CV_2$ and any of its subalgebras.
\end{rmk}
%%%%%%

%%%%%%
\begin{rmk}
    The normally-ordered product of two Hall--Littlewood operators need not be Hall--Littlewood, and similarly for the normally-ordered product of two anti-Hall--Littlewood operators, due to the fact the $\mathfrak{R}$-grading on the vector space $\CV$ is only a vertex-algebraic filtration.
\end{rmk}
%%%%%%

%%%%%%
\begin{prop}\label{prop:HLchiralring}
    Let $\CV$ be a graded-unitary vertex algebra whose internal $\Z$ grading satisfies the BPS bound. The normally-ordered product of Hall--Littlewood operators in the associated graded PVA ${\rm gr}_\mathfrak{R} \CV$ is also a Hall--Littlewood operator. Moreover, the vertex Lie algebra action of a Hall--Littlewood operator on another is again a Hall--Littlewood operator.
\end{prop}
%%%%%%

%%%%%%
\begin{proof}
    The first assertion follows from the fact that the normally-ordered product inside ${\rm gr}_\mathfrak{R} \CV$ of a vector in tri-degree $(h, R, d)$ and a vector in tri-degree $(h',R',d')$ is a vector tri-degree $(h + h',R + R', d + d')$. The second assertion follows from the fact that the vertex Lie algebra action of a vector in tri-degree $(h, R, d)$ on a vector in tri-degree $(h',R',d')$ is a vector in tri-degree $(h + h'-1, R + R'-1, d + d')$; the BPS bound ensures that only the leading term in the vertex Lie algebra action is nonzero as all other terms have the same $\mathfrak{R}$-degree but lower conformal weight.
\end{proof}
%%%%%%

This proposition implies that the subspace of Hall--Littlewood operators naturally form a Poisson algebra, leading to the following definition.

%%%%%%
\begin{dfn}
    Let $\CV$ be a graded-unitary vertex algebra $\CV$ with internal $\Z$ grading that satisfies the BPS bound. The Poisson algebra ${\rm HL}_\CV$ of Hall--Littlewood operators is called the \defterm{Hall--Littlewood chiral ring} of $\CV$. The subalgebra of Higgs branch operators is called the \defterm{Higgs branch chiral ring}.
\end{dfn}
%%%%%%

%%%%%%
\begin{rmk}
    The Hall--Littlewood chiral ring ${\rm HL}_\CV$ is naturally bigraded by the filtration degree $R$ and the additional $\Z$ grading $d$; the commutative product has bidegree $(0,0)$ and the Poisson bracket has bidegree $(-1,0)$. The Higgs branch chiral ring is naturally a Poisson subalgebra of the Hall--Littlewood chiral ring.
\end{rmk}
%%%%%%

%%%%%%
\begin{physrmk}
    In the context of the four-dimensional SCFT/VOA correspondence, the Higgs branch chiral ring is the algebra of functions on the Higgs branch of the parent $\CN=2$ SCFT. This is expected to be a Poisson manifold with finitely many symplectic leaves and is often a symplectic singularity in the sense of Beauville \cite{beauville1999symplectic}. The Hall--Littlewood chiral ring generally has nilpotent elements, but its nilradical is expected to be concentrated in degree $d > 0$.
\end{physrmk}
%%%%%%

%%%%%%
\begin{rmk}
    There is an analogous proposition for anti-Hall--Littlewood operators and correspondingly a second Poisson algebra $\overline{\rm HL}_\CV$ called the \defterm{Hall--Littlewood anti-chiral ring}. We will mostly focus on the Hall--Littlewood chiral ring as there is no meaningful distinction between the two chiral rings at this level of formality, and in cases of interest the two are isomorphic.
\end{rmk}
%%%%%%

The following result is immediate.

%%%%%%
\begin{lem}\label{lem:HLSbSf}
    The Hall--Littlewood chiral ring of ${\rm Sb}[V]$ is the ring of (polynomial) functions on $V$
    %%%%%%
    \[
        {\rm HL}_{{\rm Sb}[V]} \simeq \C[V]
    \]
    %%%%%%
    equipped with the its natural Poisson bracket coming from the symplectic structure on $V$. The Hall--Littlewood chiral ring of ${\rm Sf}[\C^2 \otimes U]$ is the exterior algebra on $U$ or, equivalently, the ring of functions on the parity-shifted dual $\Pi U^*$
    %%%%%%
    \[
        {\rm HL}_{{\rm Sf}[\C^2 \otimes U]} \simeq \bigwedge\!{}^\bullet U \simeq \C[\Pi U^*]
    \]
    %%%%%%
    equipped with the trivial Poisson structure.
\end{lem}
%%%%%%

Let us now return to the main cases of interest. Namely, when $\CV = (\CV_{\rm matter} \otimes {\rm Sf}[\C^2 \otimes \fg])^G$ is the graded-unitary vertex algebra underlying the relative semi-infinite cohomology of a graded-unitary vertex algebra $\CV_{\rm matter}$ with a good Hamiltonian $\wh{\fg}_{-2h^\vee}$ action. If $\CV_{\rm matter}$ has an internal $\Z$ grading satisfying the BPS bound, then $\CV$ is graded by $\Z^2$ with one factor coming from $\CV_{\rm matter}$ and the second coming from ${\rm Sf}[\C^2 \otimes \fg]$. We equip $\CV$ with a $\Z$ grading satisfying the BPS bound by taking the diagonal grading. The following result identifies the relation between the Poisson algebra ${\rm HL}_{\CV_{\rm matter}}$ and the (DG) Poisson algebra ${\rm HL}_\CV$.

%%%%%%
\begin{prop}\label{prop:HLBRST}
    Let $\CV_{\rm matter}$ be a graded-unitary vertex algebra equipped with a good Hamiltonian $\wh{\fg}_{-2h^\vee}$ action. Moreover, assume its internal $\Z$ grading satisfies the BPS bound and let $\CV = (\CV_{\rm matter} \otimes {\rm Sf}[\C^2 \otimes \fg])^G$ be the corresponding graded-unitary vertex algebra with the above $\Z$ grading satisfying the BPS bound. Then ${\rm HL}_\CV$ can be identified with the derived Poisson reduction of ${\rm HL}_{\CV_{\rm matter}}$ by $G$. In other words,
    %%%%%%
    \[
        {\rm Spec}~{\rm HL}_{\CV} \simeq ({\rm Spec}~{\rm HL}_{\CV_{\rm matter}})/\!\!/ G~.
    \]
    %%%%%%
\end{prop}
%%%%%%

%%%%%%
\begin{rmk}
    The differential $Q^+$ (resp. $Q^-$) on ${\rm gr}_\mathfrak{R} \CV$ has grading $(h, R, d) = (0, \frac{1}{2}, 1)$ (resp. $(h, R, d) = (0, \frac{1}{2}, -1)$) and hence preserves the space of Hall--Littlewood (resp. anti\=/Hall--Littlewood) operators. If one views $d$ as a cohomological grading on ${\rm HL}_\CV$, $Q^+$ naturally equips the Hall--Littlewood chiral ring with the structure of a DG $0$-shifted Poisson algebra. On the other hand, one could instead view $2R$ as a cohomological grading on ${\rm HL}_\CV$, leading to a DG $2$-shifted Poisson algebra.
\end{rmk}
%%%%%%

%%%%%%
\begin{proof}
    The statement at the level of Poisson algebras follows from the fact that the Hall--Littlewood chiral ring of $\CV$ is simply the $G$-invariants in the Hall--Littlewood chiral ring of $\CV_{\rm matter} \otimes {\rm Sf}[\C^2 \otimes \fg]$ together with Lemma \ref{lem:HLSbSf}:
    %%%%%%
    \[
        {\rm HL}_{\CV} \simeq \left({\rm HL}_{\CV_{\rm matter}} \otimes \bigwedge\!{}^\bullet \fg\right)^G~.
    \]
    %%%%%%
    The theorem then follows upon making explicit the action of $Q^+$ on Hall--Littlewood operators in $\CV$. First note that goodness of the action implies that the currents $J_A$, which have $h = 1$, $p = 1$, and $d = 0$, are naturally Higgs branch operators in $\CV_{\rm matter}$ and hence Hall--Littlewood operators. Moreover, Poisson bracket with $J_A$ generates the action of $G$ on ${\rm HL}_{\CV_{\rm matter}}$, \emph{i.e.} they are the components of the comoment map for the $G$ action.
	
    The differential $Q^+$ takes the form
    %%%%%%
    \[
    \begin{aligned}
        Q^+ & = \sum_{n > 0} \tfrac{1}{n}\bigg(\eta^{+, A}_{-n} J^{[0]}_{A,n} - J^{[1]}_{A,-n} \eta^{+, A}_{n}\bigg)\\
		& \qquad + \sum_{m,n > 0}f_{ABC} \left(\tfrac{1}{m(m+n)}\eta^{+, A}_{-n} \eta^{-,B}_{-m}\eta^{+,C}_{n+m} - \tfrac{1}{2mn}\eta^{+, A}_{-n} \eta^{+,B}_{-m}\eta^{-,C}_{n+m}\right)~,
    \end{aligned}
    \]
    %%%%%%
    where $\eta^{\pm, A}_{-n}$, $J^{[1]}_{A,-n}$ are understood as the endomorphisms describing the normally-ordered product with $\eta^{\pm, A}$ and $J_A$ and their derivatives in ${\rm gr}_\mathfrak{R} \CV$ and $\eta^{\pm, A}_{n}$, $J^{[0]}_{A,n}$ are understood as the endomorphisms appearing in the vertex Lie structure applied to $\eta^{\pm, A}$ and $J_A$. As a result of the BPS bound, the only term in this expression that is non-zero when acting on Hall--Littlewood operators is $-J^{[1]}_{A,-1} \eta^{+, A}_{1}$. In particular, we find the action of $Q^+$ to be induced by the restriction to $G$-invariants of the following action,
    %%%%%%
    \[
    \begin{aligned}
        &Q^+ x  = 0~, \qquad\qquad x \in {\rm HL}_{\CV_{\rm matter}}\\
        &Q^+ \eta^{-, A} = -K^{AB} J_B~,
    \end{aligned}
    \]
    %%%%%%
    which indeed reproduces the $G$-invariant subcomplex of the Koszul complex restricting to the zero level set of the moment map.
\end{proof}
%%%%%%

In light of Corollary \ref{cor:ddcPVA}, the graded-unitary structure on $\CV$ implies that these derived Poisson reductions are formal.
%%%%%%
\begin{cor}\label{cor:HLformality}
    Let $\CV_{\rm matter}$ be a graded-unitary vertex algebra whose internal $\Z$ grading satisfies the BPS bound. Moreover, assume that it is equipped with a good Hamiltonian $\wh{\fg}_{-2h^\vee}$ action. Then the derived Poisson reduction of ${\rm HL}_{\CV_{\rm matter}}$ by $G$ is formal.
\end{cor}
%%%%%%

%%%%%%
\begin{rmk}
    As an explicit example, we can take $\CV_{\rm matter}$ to be the symplectic boson VOA ${\rm Sb}[\CR]$ based on a pseudo-real representation $\CR$ of a compact, semisimple group $G_c$ with
    %%%%%%
    \[
        \Tr_\CR(X^2) = 2 \Tr_{\rm ad}(X^2)~.
    \]
    %%%%%%
    for all $X \in \fg_c$. This condition ensures that ${\rm Sb}[\CR]$ together with its graded-unitary structure given in Section \ref{sec:examples} admits a good Hamiltonian $\wh{\fg}_{-2h^\vee}$ action. This corollary then says that the derived Poisson reduction $\CR/\!\!/G$ is formal.
\end{rmk}
%%%%%%

\subsection*{Acknowledgments}
The authors would particularly like to acknowledge Dylan Butson for key suggestions during the early stages of this work. The authors also thank Kevin Costello, Justin Hilburn, Leonardo Rastelli, and Pavel Safronov for useful discussions on related topics. The work of CB is supported in part by grant \#494786 from the Simons Foundation, by ERC Consolidator Grant \#864828 “Algebraic Foundations of Supersymmetric Quantum Field Theory” (SCFTAlg), and by the STFC consolidated grant ST/T000864/1. The work of NG is supported by ERC Consolidator Grant \#864828 “Algebraic Foundations of Supersymmetric Quantum Field Theory” (SCFTAlg).

\bibliographystyle{amsalpha}
\bibliography{formality}
	
\end{document}